\documentclass[a4paper,10pt]{amsart}
\usepackage{amsmath,amsthm,amssymb,enumerate,color}

\newcommand{\norm}[1]{\left\Vert#1\right\Vert}
\newcommand{\abs}[1]{\left\vert#1\right\vert}
\newcommand{\set}[1]{\left\{#1\right\}}
\newcommand{\Real}{\mathbb{R}}

\newcommand\bR{\mathbb{R}}

\newcommand\bF{\mathbb{F}}

\newcommand\fra{\mathfrak{a}}
\newcommand\frb{\mathfrak{b}}
\newcommand\frc{\mathfrak{c}}

\newcommand\cB{\mathcal{B}}

\newcommand\cF{\mathcal{F}}

\newcommand\cI{\mathcal{I}}
\newcommand\cJ{\mathcal{J}}
\newcommand\cK{\mathcal{K}}
\newcommand\cL{\mathcal{L}}
\newcommand\cM{\mathcal{M}}
\newcommand\cN{\mathcal{N}}

\newcommand\cP{\mathcal{P}}

\newcommand{\<}{\langle}
\renewcommand{\>}{\rangle}

 \newtheorem{thm}{Theorem}[section]
 \newtheorem{cor}[thm]{Corollary}
 \newtheorem{lem}[thm]{Lemma}
 
 \theoremstyle{definition}
 \newtheorem{defn}[thm]{Definition}
 \theoremstyle{remark}
 \newtheorem{rem}[thm]{Remark}
 \newtheorem*{ex}{Example}

\newtheorem{assumption}{Assumption}[section]

\numberwithin{equation}{section}

\author[I. Gy\"ongy]{Istv\'an Gy\"ongy}
\address{School of Mathematics and Maxwell Institute \\
          University of Edinburgh \\
          King's Buildings \\
          EH9 3JZ Edinburgh, United Kingdom}
\email{gyongy@maths.ed.ac.uk}

\author[P. R. Stinga]{Pablo Ra\'ul Stinga}
\address{Department of Mathematics \\
          The University of Texas at Austin \\
          1 University Station, C1200\\
          Austin, TX 78712-1202, United States of America}
\email{stinga@math.utexas.edu}

\thanks{The second author was supported by Ministerio de Ciencia e Innovaci\'{o}n de Espa\~{n}a MTM2008-06621-C02-01 and Grant COLABORA 2010/01 from Planes Riojanos de I+D+I}



\begin{document}

%%%%%%%%%%%%%%%%%%%%%%%%%%%%%%%%%%%%%%%%%%%%%%
\title[Rate of convergence]{Rate of convergence of Wong-Zakai approximations \\ for stochastic partial differential equations}
%%%%%%%%%%%%%%%%%%%%%%%%%%%%%%%%%%%%%%%%%%%%%%

%%%%%%%%%%%%%%%%%%%%%%%%%%%%%%%%%%%%%%%%%%%%%%
\begin{abstract}
In this paper we show that the rate of convergence of Wong-Zakai approximations for stochastic partial differential equations driven by Wiener processes is essentially the same as the rate of convergence of the driving processes $W_n$ approximating the Wiener process, provided the area processes of $W_n$  also converge to those of $W$ with that rate. We consider non-degenerate and also degenerate stochastic PDEs with time dependent coefficients.
\end{abstract}
%%%%%%%%%%%%%%%%%%%%%%%%%%%%%%%%%%%%%%%%%%%%%%

%%%%%%%%%%%%%%%%%%%%%%%%%%%%%%%%%%%%%%%%%%%%%%
\maketitle
%%%%%%%%%%%%%%%%%%%%%%%%%%%%%%%%%%%%%%%%%%%%%%

%%%%%%%%%%%%%%%%%%%%%%%%%%%%%%%%%%%%%%%%%%%%%%%%%%%%%%
\section{Introduction}
%%%%%%%%%%%%%%%%%%%%%%%%%%%%%%%%%%%%%%%%%%%%%%%%%%%%%%

Consider for each integer $n\geq1$ the stochastic PDE
\begin{equation}                                                                               \label{eq for un}
du_n(t,x)=(L_nu_n(t,x)+f_n)\,dt+(M_n^ku_n(t,x)+g_n^k)\,dW^k_n(t),
\end{equation}
for $(t,x)\in H_T=(0,T]\times\bR^d$, for a fixed $T>0$, with initial condition
\begin{equation}                                                                                \label{ini}
u_n(0,x)=u_{n0}(x), \quad x\in\Real^d,
\end{equation}
given on a probability space $(\Omega,\cF,P)$,
where $L_n$ and $M_n^k$ are second
and first order differential operators in $x\in\bR^d$, respectively
for every $\omega\in\Omega$. The free terms,
$f_n$ and $g_n=(g_n^k)$ are random fields,
and $W_n=(W_n^k)$ is a continuous
$d_1$-dimensional stochastic process
with finite variation over $[0,T]$, for $k=1,\ldots,d_1$.

Unless otherwise stated we use the summation convention with respect to repeated indices throughout the paper. The summation convention is not used
if the repeated index is the subscript $n$.

 The operators $L_n$, $M_n^k$ are of the form
$$
L_n=a_n^{ij}(t,x)D_iD_j+a_n^i(t,x)D_i+a_n(t,x),
\qquad M_n^k=b^{ik}_n(t,x)D_i+b_{n}^{k}(t,x),
$$
where $a_n^{ij},\ldots,b_{n}^{k}$ are real-valued
bounded
functions on $\Omega\times[0,T]\times\bR^d$
for all $i,j=1,...,d$, $k=1,...,d_1$, and integers $n\geq1$,
$D_i=\frac{\partial}{\partial x^i}$ for $i=1,2,...,d$, and  $x^i$ is the
$i$-th co-ordinate of $x\in\bR^d$. The free terms, $f_n$, $g_n^1$,...,$g^{d_1}_n$ are
real-valued functions on
$\Omega\times[0,T]\times\bR^d$ for each $n$.
We assume that $L_n$ is either uniformly elliptic or degenerate elliptic for all $n$.

Assume that
the operators $L_n$, $M_n^k$, the free terms
$f_n$, $g_n^k$ and the initial data $u_{n0}$
converge to some operators
\begin{equation}                                                                                      \label{operators L and M}
L=a^{ij}(t,x)D_{ij}+a^i(t,x)D_i+a(t,x),\quad M^k=b^{ik}(t,x)D_i+b^k(t,x),
\end{equation}
random fields $f$, $g^k$ and initial data
$u_0$ respectively, and $W_n(t)$ converges to
a $d_1$-dimensional Wiener process
in probability, uniformly in $t\in[0,T]$.
Then
under some smoothness conditions on the coefficients of
$L_n$, $L$, $M_n^k$, $M^k$ and on the data $u_{0n}$, $f_n$, $g_n^k$,
$u_0$, $f$, $g^k$, and under some
additional conditions
on the convergence of the related area processes
and on the growth of the auxiliary process
$B_n$ (defined in \eqref{arean} and \eqref{def Bn} below),
the solution $u_n$ to \eqref{eq for un}
converges in probability to a random field $u$
that satisfies the stochastic PDE
\begin{equation}                                                                             \label{eq}
du(t,x)=(Lu(t,x)+f)\,dt+(M^ku(t,x)+g^k)\,\circ dW^k(t),\quad(t,x)\in H_T
\end{equation}
with initial condition
\begin{equation}                                                                              \label{ini0}
u(0,x)=u_0(x),     \quad x\in\bR^d,
\end{equation}
where `$\circ$' indicates the Stratonovich differential.
(See, e.g., \cite{G1} and \cite{G2}.) When $M^k$ and $g^k$
do not depend on the variable $t$, then
$$
(M^ku(t,x)+g^k)\circ \,dW^k=\tfrac{1}{2}(M^kM^ku(t,x)+M^kg^k)\,dt
+(M^ku(t,x)+g^k)\,dW^k.
$$

One of the important questions in the analysis of approximation schemes is the estimation of the speed of convergence.  In this paper we show that,
if the continuous finite variation processes  $W_n$ and their
area processes converge almost surely to a Wiener
process $W$ and to its area processes, respectively, with a given rate,
then $u_n(t)$ converges almost surely with {\it essentially} the same rate.
 The results of this paper are motivated by a question
about robustness of nonlinear filters for partially observed processes,
$(X(t),Y(t))_{t\in[0,T]}$. For a large class of
{\it signal and  observation} models, the signal $X$ and the observation
$Y$ are governed by stochastic differential
equations with respect to Wiener processes, and a basic assumption is that
the signal process is a non-degenerate It\^o process. Thus the signal is modelled
by a process, which has infinite (first) variation on any (small) finite interval.
In practice, however, due to the smoothing effect of measurements,
the ``signal data" is a process which has finite variation on any finite interval.
This process can be viewed as an approximation $Y_n$ to $Y$,
and it is natural to assume that $Y_n$ and its area processes
converge almost surely in the sup norm to $Y$ and its area processes, with some speed.
By a direct application of the main theorems of the present article one can show that
the ``robust filtering equation",
with $Y_n$ in place of $Y$, admits a unique solution $p_n$ which
converges
almost surely with almost the same order to the conditional density of
$X(t)$ given the observation
$\{Y(s): s\in[0,t]\}$.
The filtering equations in case of correlated signal and observation noise
are stochastic PDEs with coefficients depending
on the observations. Thus approximating the observations we approximate
also the differential operators in the stochastic PDEs. This is why we consider
equation \eqref{eq for un} with random operators $L_n$ and $M^k_n$ depending
also on $n$, the parameter of the approximation.

Our results improve and generalise the results
of \cite{Gyongy-Shmatkov} and \cite{Shmatkov}, where only half of the
order of convergence  of $W_n$ is obtained
for the order of convergence of $u_n$.
Moreover, our conditions are weaker,
and we prove the optimal rate also in the case of
degenerate stochastic PDEs, which allows to get our rate of convergence
result also in the case of degenerate signal and observation models.

Wong-Zakai  approximations of stochastic PDEs were
 studied intensively
in the literature. See, for example, \cite{AT}-\cite{G2},
\cite{Gyongy-Shmatkov}-\cite{H}, \cite{Shmatkov}-\cite{TZ},
and the references therein. With the exception of \cite{BF},
\cite{H}, \cite{Gyongy-Shmatkov} and \cite{Shmatkov}
 the papers above
prove convergence results of Wong-Zakai approximations
for stochastic PDEs with various generalities, but do not
present rate of convergence estimates. Wong-Zakai type
approximation results for semilinear and fully nonlinear
SPDEs are obtained via rough path approach in
\cite{FO1}--\cite{FO3}.

In \cite{BF},
the initial value problem
\eqref{eq for un}--\eqref{ini} is considered with non-random
coefficients and without
free terms,
when $W_n$ are polygonal approximations to
the Wiener process $W$.
By the method of characteristics it is proved that
$u_n(t,x)$ converges almost surely, uniformly
in $(t,x)\in[0,T]\times\bR^d$.
Though the rate of convergence
of $u_n$ to $u$ is not stated explicitly in \cite{BF},
from the rate of convergence result proved in \cite{BF}
for the characteristics,
one can easily deduce that for every $\kappa<1/4$ there
exists a finite random variable $\xi_{\kappa}$ such that
almost surely $|u_n(t,x)-u(t,x)|\leq \xi n^{-\kappa}$
for all $t\in[0,T]$ and  $x\in\bR^d$.
We note that for polygonal approximations
the almost sure order of convergence of $W_n$ and its area processes
are of order $\kappa<1/2$, and thus by our paper the almost sure
rate of convergence of the Wong-Zakai approximations
is the same $\kappa<1/2$, in Sobolev norms, and via
Sobolev's embedding in the supremum norm as well.
In \cite{H} the rate of convergence of Wong-Zakai approximations
of stochastic PDEs driven by Poisson random measures is investigated.

Let us conclude with introducing some notation used throughout the paper.
All random objects are given on a fixed
probability space $(\Omega,\cF,P)$ equipped with a right-continuous filtration
$\bF=(\cF_t)_{t\geq0}$, such that $\cF_0$ contains all the $P$-null
sets of the complete $\sigma$-algebra $\cF$.
The $\sigma$-algebra of predictable subsets of $[0,\infty)\times\Omega$
is denoted by $\cP$ and the
$\sigma$-algebra of Borel subsets of $\bR^d$ is denoted by $\cB(\bR^d)$.
The notation $C_0^{\infty}=C_0^{\infty}(\bR^d)$ stands for the space of real valued smooth functions
with compact support on $\bR^d$.
For an integer $m$ we use the notation $H^m$ for the Hilbert-Sobolev space
$W^m_2(\bR^d)$
of generalised functions on $\bR^d$. For $m\geq 0$ it is the closure of
$C_0^{\infty}$
 in the norm $|\cdot|_m$ defined by
$$
|f|_m^2=\sum_{|\alpha|\leq m}\int_{\bR^d}|D^{\alpha}f(x)|^2\,dx,
$$
where $D^{\alpha}=D^{\alpha_1}_1D^{\alpha_2}_2\cdots D^{\alpha_d}_d$
for multi-indices $\alpha=(\alpha_1,\alpha_2,\alpha_d)\in\{0,1,\ldots\}^d$,
and $D^0$
is the identity operator. For $m<0$,  $H^m$ is the closure of $C_0^{\infty}$
in the norm
$$
|f|_{m}=\sup_{g\in C^\infty_0, |g|_{-m}\leq 1}(f,g)_0,
$$
where $(f,g)_0$ denotes the inner product in $L_2=H^0$.
We define in the same way the Hilbert-Sobolev space
$H^m=H^m(\bR^l)$ of $\bR^l$-valued functions $g=(g^1,\ldots,g^l)$
on $\bR^d$, such that $|g|^2_m=\sum_{k=1}^l|g^k|^2_m$.
We use the notation $(\cdot\,,\cdot)_m$ for the inner product
in $H^m$, and for $m=0$ we often use the notation $(\cdot\,,\cdot)$
instead of $(\cdot\,,\cdot)_0$. For $m\geq0$ denote by $\<\cdot\,,\cdot\>_m$ the duality
product between $H^{m+1}$ and $H^{m-1}$, based
on the inner product $(\cdot\,,\cdot)_m$
in $H^m$. For real numbers $A$ and $B$
we set $A\vee B=\max\set{A,B}$
and $A\wedge B=\min\set{A,B}$.
For sequences of random variables
$(a_n)_{n=1}^{\infty}$ and $(b_n)_{n=1}^{\infty}$
the notation $a_n=o(b_n)$ means the existence of
a sequence of random variables $\xi_n$ converging
almost surely to zero such that almost surely
$|a_n|\leq\xi_n|b_n|$ for all $n$. The notation $a_n=O(b_n)$
means the existence of a finite random variable $\eta$
such that almost surely $|a_n|\leq \eta |b_n|$ for all $n$.

\section{Formulation of the results}                                                        \label{Section:Formulation}

Let $W=(W(t))_{t\in[0,T]}$ be a $d_1$-dimensional
Wiener martingale with respect
to $\bF$, and consider for every integer $n\geq1$
an $\bR^{d_1}$-valued
$\cF_t$-adapted continuous process
$W_n=(W_n(t))_{t\in[0,T]}$ of finite
variation. Define the area processes of $W$ and $W_n$ as
\begin{equation}                                                                                                         \label{area}
A^{ij}(t):=\frac{1}{2}\int_0^tW^i(s)\,dW^j(s)-W^j(s)~dW^i(s)
,\quad i,j=1,2,...,d_1,
\end{equation}
\begin{equation}\label{arean}
A_n^{ij}(t):=\frac{1}{2}
\int_0^tW^i_n(s)~dW^j_n(s)-W^j_n(s)~dW^i_n(s),
\quad i,j=1,2,...,d_1,
\end{equation}
and also the process
\begin{equation}                                                                                                         \label{def Bn}
B_n^{ij}(t):=\int_0^t(W^i(s)-W^i_n(s))\,dW^j_n(s),\qquad i,j=1,\ldots,d_1,
\end{equation}
that will play a crucial role.
We denote by $\norm{q}(t)$ the first variation of a process $q$
over the interval $[0,t]$ for $t\leq T$.

Let $\gamma>0$ be a fixed real number and assume that the following
conditions hold.

\begin{assumption}                                                                                   \label{assumption 15.04.06}
For each $\kappa<\gamma$ almost surely
\begin{enumerate}
    \item[(i)] $\sup_{t\leq T}\abs{W(t)-W_n(t)}=O(n^{-\kappa})$,
    \item[(ii)] $\sup_{t\leq T}\abs{A^{ij}(t)-A^{ij}_n(t)}=O(n^{-\kappa})$
 \quad $i\neq j$,
    \item[(iii)] $||B_n^{ij}||(T)=o(\ln n)$ for all $i,j=1,\ldots,d_1$.
\end{enumerate}
\end{assumption}
The following remark is shown in \cite{Gyongy-Shmatkov}.
\begin{rem}                                                                                                 \label{Rem:Sn and An}
Define the matrix-valued process $S_n=(S_n^{ij}(t))$, $t\in[0,T]$ by
\begin{align}
   S_n^{ij}(t)
   &= \int_0^t(W^{i}(s)-W^{i}_n(s))\,dW^j_n(s)-\frac{1}{2}\<W^i,W^j\>(t)           \label{Sn with qn} \\
   &= \int_0^t(W^{i}(s)-W^{i}_n(s))~dW^j_n(s)-\frac{1}{2}\delta_{ij}t,
   \qquad i,j=1,2,...,d_1,  \nonumber
\end{align}
for each integer $n\geq1$, where $\<W^i,W^j\>$ denotes the quadratic covariation
process of $W^i$ and $W^j$, and $\delta_{ij}=1$ for $i=j$ and it is zero
otherwise.
Then by It\^o's formula for $q^{ij}_n:=(W^i-W_n^i)(W^j-W_n^j)$ we have
$$
S_n^{ij}(t)+S_n^{ji}(t)=q^{ij}_n(0)-q^{ij}_n(t)+R^{ij}_n(t)+R^{ji}_n(t)
$$
with
$$
R^{ij}_n(t):=\int_0^t(W^i(s)-W_n^i(s))\,dW^j(s).
$$
Moreover, given Part $\mathrm{(i)}$ in Assumption \ref{assumption 15.04.06},
Part $\mathrm{(ii)}$ is equivalent to condition $\mathrm{(ii')}$:
\begin{equation}                                                                                            \label{7.22.1.12}
\sup_{t\leq T}|S_n^{ij}(t)|=O(n^{-\kappa})
\quad\text{(a.s.) for each $\kappa<\gamma$,   for $i,j=1,...,d_1$}.
\end{equation}
\end{rem}

\smallskip
Assumption
\ref{assumption 15.04.06} holds for
a large class of approximations $W_n$ of $W$.
The main examples are the following.
\begin{ex} (Polygonal approximations) Set
$W_n(t)=0$ for $t\in[0, T/n)$ and
$$
W_n(t) =
W({t_{k-1}})+n(t-t_k)(W({t_k})-W({t_{k-1}}))/T
$$
for $t\in[t_k,t_{k+1})$, where $t_k:=kT/n$
for integers $k\geq0$.
\end{ex}
\begin{ex}(Smoothing) Define
$$
W_n(t)=\int_0^1 W({t- u/n})\,du, \quad t\geq0,
$$
where $W(s):=0$ for $s<0$.
\end{ex}
One can prove, see \cite{GS}, that these examples satisfy the conditions
of Assumptions \ref{assumption 15.04.06} with $\gamma=1/2$.

Now  we formulate the conditions on the operators
$L_n$, $M_n^k$ and their convergence to operators $L$ and $M^k$.
We fix an integer $m\geq0$ and a real number $K\geq0$.

\begin{assumption}[ellipticity]                                        \label{assumption 06.02.06}
There exists a constant $\lambda\geq0$
such that for each integer $n\geq1$ for
$dP\times dt\times dx$ almost all
$(\omega,t,x)\in\Omega\times[0,T]\times \bR^d$
$$
a_n^{ij}(t,x)z^iz^j\geq \lambda |z|^2,
$$
for all $z=(z^1,z^2,...,z^d)\in\bR^d$.
\end{assumption}

If $\lambda>0$ then we need the following conditions on the
regularity of the coefficients
$\mathbf{a}_n=(a_n^{ij},a^i_n,a_n: i,j=1,...,d)$,
$\mathbf{b}_n=(b^{ik}_{n},b^k_n:i=1,...,d;k=1,...,d_1)$,
$\mathbf{a}:=(a^{ij},a^{i},a:i,j=1,...,d)$,
$\mathbf{b}:=(b^{ik},b^k:i=1,...,d;k=1,...,d_1)$
for all $n\geq1$, and on the data
$u_{n0}$, $f_n$ and $g_n=(g_n^k)$, $u_0$,
$f$, $g=(g^k)$.

\begin{assumption}                                                                          \label{assumption 06.02.07}
The coefficients
$\mathbf{a}_n$,
$\mathbf{b}_n$ and their derivatives in $x$
up to order $m+4$ are
$\cP\times\cB(\bR^d)$-measurable functions, and they are
in magnitude bounded by $K$. For each $n\geq1$,
 $f_n$ is an $H^{m+3}$-valued predictable process, $g_n=(g_n^k)$
 is an $H^{m+4}(\bR^{d_1})$-valued predictable process and
 $u_{n0}$ is an $H^{m+4}$-valued $\cF_0$-measurable random variable,
such that
for every $\varepsilon>0$ almost surely
$$
|u_{n0}|_{m+3}=O(n^{\varepsilon}),
\quad \int_0^T|f_n|_{m+3}^2\,dt=O(n^{\varepsilon}),
\quad
\sup_{t\leq T}|g_n(t)|_{m+4}=O(n^{\varepsilon}).
$$
\end{assumption}

One knows, see Theorem \ref{theorem 7.22.1} below,
that if Assumption \ref{assumption 06.02.06}
with $\lambda>0$ and Assumption
\ref{assumption 06.02.07} hold, then for each $n\geq1$
there is a unique generalised solution $u_n$ to
\eqref{eq for un}-\eqref{ini}.

\begin{assumption}                                                                              \label{assumption 1.22.1.12}
The coefficients $\mathbf{a}$
and $\mathbf{b}$
and their derivatives in $x$ up to order $m+1$
are $\cP\times\cB(\bR^d)$-measurable functions
on $\Omega\times H_T$, and they are in magnitude bounded by $K$.
The initial value $u_0$ is an $H^{m+1}$-valued $\cF_0$-measurable random
variable,
$f$ is an $H^{m}$-valued predictable processes and $g=(g^k)$
is an $H^{m+1}(\bR^{d_1})$-valued predictable process such that
almost surely
$$
\int_0^T|f(t)|_{m}^2\,dt+\sup_{t\leq T}|g(t)|^2_{m+1}<\infty.
$$
\end{assumption}

\begin{assumption}                                                                                 \label{assumption 06.02.08}
We have
$$
\sup_{H_T}|D^\alpha\mathbf{a}_n-D^\alpha\mathbf{a}|
=O(n^{-\gamma}),
\quad
\sup_{H_T}|D^\beta\mathbf{b}_n-D^\beta\mathbf{b}|
=O(n^{-\gamma}),
$$
for all $|\alpha|\leq(m-1)\vee0$ and $|\beta|\leq m+1$, and
$$
\int_0^T| f_n(t)-f(t)|_{m-1}^2\,dt+\sup_{t\leq T}|g_n(t)-g(t)|_{m+1}^2
=O(n^{-2\gamma}).
$$
\end{assumption}
Now we formulate our  main result when $\lambda>0$
in Assumption \ref{assumption 06.02.06}, and
$b^{ik}_n$, $b^k_n$ and $g^k_n$ do not depend
on $t\in[0,T]$.
\begin{thm}                                                                                       \label{theorem main1}
Assume that ${\mathbf b}_n$ and $g_n$
are independent
of $t$. Let Assumptions \ref{assumption 15.04.06},
\ref{assumption 06.02.06} with $\lambda>0$,
\ref{assumption 06.02.07},
\ref{assumption 1.22.1.12}
and \ref{assumption 06.02.08} hold. Then almost surely
$$
\sup_{t\leq T}|u_n(t)-u(t)|^2_m+\int_0^T|u_n(t)-u(t)|^2_{m+1}\,dt
=O(n^{-2\kappa}),
\quad\text{for any $\kappa<\gamma$}.
$$
\end{thm}

In the degenerate case, $\lambda=0$,
instead of Assumptions \ref{assumption 06.02.07},
\ref{assumption 1.22.1.12}
and \ref{assumption 06.02.08} we need to impose
stronger conditions.
\begin{assumption}                                                 \label{assumption deg smooth}
(i) For each $n\geq1$ there exist functions
$\sigma_{n}^{ir}$ on $\Omega\times H_T$, for $i=1,...,d$ and $r=1,...,p$, for some
$p\geq1$,  such that
$a_n^{ij}=\sigma_{n}^{ir}\sigma_{n}^{jr}$
for all $i,j=1,\ldots,d$.  (ii)
The functions $\sigma_{n}^{ir}$,
$b_n^i$ and their derivatives in $x$ up to order $m+6$,
the functions
 $a_n^i$, $a_n$, $b_n$ and their derivatives in $x$
up to order $m+5$
are $\cP\otimes\cB(\bR^d)$-measurable  functions on $\Omega\times H_T$,
and in magnitude are bounded by $K$ for $i=1,...,d$ and $r=1,...,p$.
For each $n\geq1$,
 $f_n$ is an $H^{m+4}$-valued predictable process, $g_n=(g_n^k)$
 is an $H^{m+5}(\bR^{d_1})$-valued predictable process and
 $u_{n0}$ is an $H^{m+4}$-valued $\cF_0$-measurable random variable,
such that
for every $\varepsilon>0$
$$
|u_{n0}|_{m+4}=O(n^{\varepsilon}),
\quad \int_0^T|f_n|_{m+4}^2\,dt=O(n^{\varepsilon}),
\quad
\sup_{t\leq T}|g_n(t)|_{m+5}=O(n^{\varepsilon}).
$$
\end{assumption}

\begin{assumption}                                                                               \label{assumption 1.3.3.12}
The coefficients $\mathbf{a}$
and $\mathbf{b}$
and their derivatives in $x$ up to order $m+2$
are $\cP\times\cB(\bR^d)$-measurable functions
on $\Omega\times H_T$, and they are in magnitude bounded by $K$.
The initial value $u_0$ is an $H^{m+2}$-valued $\cF_0$-measurable random
variable,
$f$ is an $H^{m+2}$-valued predictable process and $g=(g^k)$
is an $H^{m+3}(\bR^{d_1})$-valued predictable process such that
$$
\int_0^T|f(t)|_{m+2}^2\,dt+\sup_{t\leq T}|g(t)|^2_{m+2}<\infty.
$$
\end{assumption}

\begin{assumption}                                                      \label{assumption 06.02.09}
We have
$$
\sup_{H_T}|D^\alpha\mathbf{a}_n-D^\alpha\mathbf{a}|
=O(n^{-\gamma}),
\quad \sup_{H_T}|D^\beta\mathbf{b}_n-D^\beta\mathbf{b}|
=O(n^{-\gamma})
$$
for all $|\alpha|\leq m$ and $|\beta|\leq m+1$, and
$$
\int_0^T| f_n-f|_m^2\,dt
=O(n^{-\gamma}),
\quad\int_0^T|g_n- g|_{m+1}^2\,dt
=O(n^{-2\gamma}).
$$
\end{assumption}

\begin{rem}
Notice that Assumption \ref{assumption deg smooth} (i) implies
Assumption \ref{assumption 06.02.06} with $\lambda=0$.
\end{rem}

\begin{thm}                                                                       \label{theorem main2}
Assume that ${\mathbf b}_n$ and $g_n$  do not depend on $t$.
Let Assumptions  \ref{assumption 15.04.06},
  \ref{assumption deg smooth}, \ref{assumption 1.3.3.12}
and \ref{assumption 06.02.09} hold. Then
$$
\sup_{t\leq T}\abs{u_n-u}_m
=O(n^{-\kappa}) \quad\text{a.s. for
each $\kappa<\gamma$}.
$$
\end{thm}

Let us now consider the case when all the coefficients and free terms
may depend on $t\in[0,T]$. We use the notation
$$
h_n:=\left(b_n^{ik}, b_n^k, g_n^k:i=1,...,d, \,k=1,...,d_1\right), \quad n\geq1,
$$
$$
h:=\left(b^{ik}, b^k, g^k:i=1,...,d, \,k=1,...,d_1\right).
$$

We make the following assumption.

\begin{assumption}                                                         \label{assumption 0.5.3.12}
There exist
$\cP\mathcal\otimes\cB(\bR^d)$-measurable bounded functions
$$
\mathbf b_n^{(r)}=(b^{ik(r)}_n,\, b^{k(r)}_n:i=1,...,d, \,k=1,...,d_1), \quad r=0,...,d_1, \,n\geq1
$$
and $H^0(\bR^{d_1})$-valued bounded predictable processes
$g_n^{(r)}=(g_n^{k(r)}:k=1,...,d_1)$,  such that
$$
d(h_n(t),\varphi)=(h^{(0)}_n(t),\varphi)\,dt+(h_n^{(k)}(t),\varphi)\,dW_n^k(t),
\quad n\geq 1,
$$
where $h_n^{(r)}=(\mathbf b_n^{(r)}, g_n^{(r)})$. For $r=0,...d_1$ and
$j=1,...,d_1$ there exist
$\cP\mathcal\otimes\cB(\bR^d)$-measurable bounded functions
$$
\mathbf b^{(r)}=(b^{ik(r)},\, b^{k(r)}:i=1,...,d, \,k=1,...,d_1),
$$
$$
\mathbf b^{(jr)}=(b^{ik(jr)},\, b^{k(jr)}:i=1,...,d, \,k=1,...,d_1),
$$
and $H^0(\bR^{d_1})$-valued bounded predictable processes
$g_n^{k(r)}$ and $g_n^{k(jr)}$,  $k=1,...,d_1$,
such that
$$
d(h(t),\varphi)=(h^{(0)}(t),\varphi)\,dt+(h^{(k)}(t),\varphi)\,dW_n^k(t),
$$
$$
d(h^{(j)}(t),\varphi)=(h^{(j0)}(t),\varphi)\,dt+(h^{(jk)}(t),\varphi)\,dW^k(t)
$$
for $\varphi\in C_0^{\infty}(\bR^d)$,
for $j=1,...,d_1$, where
$h^{(r)}=(\mathbf b^{(r)}, g^{k(r)}:k=1,\ldots,d_1)$ and
$h^{(jr)}=(\mathbf b^{(jr)}, g^{k(jr)}:k=1,\ldots,d_1)$, $r=0,...,d_1$.
\end{assumption}

If Assumption \ref{assumption 06.02.06} holds with $\lambda>0$, then we
impose the following conditions.

\begin{assumption}                                                                                    \label{assumption 1.4.3.12}
For $n\geq 1$ the coefficients ${\mathbf b}_n^{(r)}$ and their derivatives in
$x$ up to order $m+3$
are $\cP\otimes \cB(\bR^d)$-measurable
functions on $\Omega\times H_T$, and they are bounded in magnitude by
$K$ for $r=0,1,...,d_1$. The functions $g_n^{k(0)}$ are $H^{m+2}$-valued,
$g^{k(j)}_n$ are $H^{m+3}$-valued predictable processes, such that
$$
\int_0^T|g^{k(0)}_n(t)|^2_{m+2}\,dt=O(n^{\varepsilon}),
\quad \sup_{H_T}|g^{k(j)}_n|_{m+3}=O(n^{\varepsilon})
$$
for each $\varepsilon>0$ and all $k,j=1, ...,d_1$.
\end{assumption}

\begin{assumption}                                                                                        \label{assumption 1.6.3.12}
The coefficients ${\mathbf b}^{(r)}$, ${\mathbf b}^{(jr)}$ and their derivatives in
$x$ up to order $m+1$
are $\cP\otimes \cB(\bR^d)$-measurable
functions on $\Omega\times H_T$, and they are bounded in magnitude by
$K$ for $r=0,1,...,d_1$ and $j=1,2,..,d_1$.
The functions $g^{k(r)}$
and $g^{k(jr)}$ are $H^{m+1}$-valued predictable processes,
and are bounded in $H^{m+1}$, for $r=0,1,...,d_1$ and $j=1,2,..,d_1$.
\end{assumption}

\begin{assumption}                                                                         \label{assumption 2.6.3.12}
For $j=1,2,...,d_1$ we have
$$
\sup_{H_T}|D^{\alpha}{\mathbf b}^{(j)}_n-D^{\alpha}{\mathbf b}^{(j)}|
=O(n^{-\gamma})
\quad\text{for $|\alpha|\leq m$},
$$
$$
\sup_{t\leq T}|g_n^{k(j)}-g^{k(j)}|_m =O(n^{-\gamma})
\quad\text{for $k=1,...,d_1$}.
$$
\end{assumption}

One knows, see \cite{G2}, that under the assumptions above
the limit $u$ of $u_n$ for $n\to\infty$ exists and satisfies
\begin{equation}                                                                             \label{eqt}
du(t,x)=(Lu(t,x)+f)\,dt+(M^ku(t,x)+g^k)\,\circ dW^k(t),\quad(t,x)\in H_T
\end{equation}
with initial condition
\begin{equation}                                                                              \label{ini0t}
u(0,x)=u_0(x),     \quad x\in\bR^d,
\end{equation}
where
\begin{align*}
 (M^ku(t,x)+g^k)\circ \,dW^k=&\tfrac{1}{2}(M^kM^ku(t,x)+M^kg^k(t,x))\,dt  \\
 &+(M^ku(t,x)+g^k(t,x))\,dW^k(t)\\
 &+\frac{1}{2}\sum_{k=1}^{d_1}(M^{k(k)}u(t,x)+g^{k(k)}(t,x))\,dt,
\end{align*}
with $M^{k(k)}:=b^{ik(k)}(t,x)D_i+b^{k(k)}(t,x)$.

We have the following results on the rate of convergence.
\begin{thm}                                                                                       \label{theorem main1t}
 Let Assumptions \ref{assumption 15.04.06},
\ref{assumption 06.02.06} with $\lambda>0$,
\ref{assumption 06.02.07},
\ref{assumption 1.22.1.12},
\ref{assumption 06.02.08} and
\ref{assumption 0.5.3.12} through \ref{assumption 2.6.3.12} hold.
Then for each $\kappa<\gamma$
$$
\sup_{t\leq T}|u_n(t)-u(t)|^2_m+\int_0^T|u_n(t)-u(t)|^2_{m+1}\,dt
=O(n^{-2\kappa}),
$$
where $u$ is the generalised solution of \eqref{eqt}-\eqref{ini0t}.
\end{thm}

Let us now consider the case when $\lambda=0$
in Assumption \ref{assumption 06.02.06}.

\begin{assumption}                                                                                    \label{assumption 8.6.3.12}
For $n\geq 1$ the coefficients ${\mathbf b}_n^{(r)}$ and their derivatives in
$x$ up to order $m+4$
are $\cP\otimes \cB(\bR^d)$-measurable
functions on $\Omega\times H_T$, and they are bounded in magnitude by
$K$ for $r=0,1,...,d_1$. The functions $g_n^{k(0)}$ are $H^{m+3}$-valued,
$g^{k(j)}_n$ are $H^{m+4}$-valued predictable processes, such that
$$
\int_0^T|g^{k(0)}(t)|^2_{m+3}\,dt=O(n^{\varepsilon}),
\quad \sup_{H_T}|g^{k(j)}|_{m+4}=O(n^{\varepsilon})
$$
for each $\varepsilon>0$ and all $k,j=1, ...,d_1$.
\end{assumption}

\begin{assumption}                                                                                        \label{assumption 9.6.3.12}
The coefficients ${\mathbf b}^{(r)}$, ${\mathbf b}^{(jr)}$ and their derivatives in
$x$ up to order $m+2$
are $\cP\otimes \cB(\bR^d)$-measurable
functions on $\Omega\times H_T$, and they are bounded in magnitude by
$K$ for $r=0,1,...,d_1$ and $j=1,2,..,d_1$.
The functions $g^{k(r)}$
and $g^{k(jr)}$ are $H^{m+1}$-valued predictable processes,
and are bounded in $H^{m+1}$, for $r=0,1,...,d_1$ and $k,j=1,2,..,d_1$.
\end{assumption}

\begin{assumption}                                                                         \label{assumption 10.6.3.12}
For $j=1,2,...,d_1$ we have
$$
\sup_{H_T}|D^{\alpha}{\mathbf b}^{(j)}_n-D^{\alpha}{\mathbf b}^{(j)}|
=O(n^{-\gamma})
\quad\text{for $|\alpha|\leq m+1$},
$$
$$
\sup_{t\leq T}|g_n^{k(j)}-g^{k(j)}|_{m+1} =O(n^{-\gamma})
\quad\text{for $k=1,...,d_1$}.
$$
\end{assumption}

\begin{thm}                                                                                        \label{theorem 5.6.3.12}
Let Assumption \ref{assumption 15.04.06},
 Assumptions \ref{assumption deg smooth} through
\ref{assumption 0.5.3.12}, and
Assumptions \ref{assumption 8.6.3.12} through
\ref{assumption 10.6.3.12}  hold.  Then
$$
\sup_{t\leq T}\abs{u_n-u}_m
=O(n^{-\kappa}) \quad\text{for
each $\kappa<\gamma$},
$$
where $u$ is the generalised solution of \eqref{eqt}-\eqref{ini0t}.
\end{thm}

\section{Auxiliaries}                                                                                     \label{section 1.5.1.12}

\subsection{Existence, uniqueness and known estimates for solutions}

Consider the equation
\begin{align}
du(t,x)=&(\cL u(t,x)+f(t,x))\,dt+(\cM^ku(t,x)+g^k(t,x))\,dW^k(t)                                   \nonumber\\
&+(\cN^{\rho}u(t,x)+h^{\rho}(t,x))\,dB^{\rho}(t), \quad t\in(0,T],
\quad x\in\bR^d                                                                                                           \label{7.22.1}
\end{align}
with initial condition
\begin{equation}                                                                                                           \label{2.5.1.12}
u(0,x)=u_0(x), \quad x\in\bR^d,
\end{equation}
where $W=(W^1,...,W^{d_1})$ is a $d_1$-dimensional Wiener martingale
with respect to $(\cF_t)_{t\geq0}$, and $B^1$,...,$B^{d_2}$
are real-valued
adapted continuous processes of finite variation over $[0,T]$.
The operators $\cL$, $\cM^k$  and $\cN^{\rho}$ are of the form
$$
\cL=\fra^{ij}D_iD_j+\fra^iD_i+\fra, \quad \cM^k=\frb^{ik}D_i+\frb^{k},\quad
\cN^{\rho}=\frc^{i\rho}D_i+\frc^{\rho},
$$
where the coefficients $\fra^{ij}$, $\fra^{i}$, $\fra$, $\frb^{ik}$, $\frb^k$,
$\frc^{i\rho}$ and $\frc^{\rho}$ are $\cP\times\cB(\bR^d)$-measurable
real-valued
bounded functions defined on $\Omega\times[0,T]\times\bR^d$
for all $i,j=1,...,d$,
$k=1,...,d_1$ and $\rho=1,...,d_2$.
The free terms $f=f(t,\cdot)$, $g^k(t,\cdot)$
and $h^{\rho}=h^{\rho}(t,\cdot)$ are $H^0$-valued
predictable processes,
and $u_0$ is an $H^1$-valued $\cF_0$-measurable random variable.

To formulate the notion of the solution we assume that
the generalised derivatives in $x$, $D_ja^{ij}$, are also bounded functions
on $\Omega\times H_T$ for all $i,j=1,...,d$.

\begin{defn}                                                            \label{Definition 2.3.3.12}
By a solution of \eqref{7.22.1}-\eqref{2.5.1.12}
we mean an $H^1$-valued weakly continuous
adapted process $u=(u(t))_{t\in[0,T]}$ , such that
\begin{align*}
    (u(t),\varphi)&=(u_0,\varphi)+\int_0^t\{-(\fra^{ij}D_iu,D_j\varphi)
+\big((\fra^i-\fra_{j}^{ij})D_iu+\fra u+f,\varphi\big)\}\,ds                                  \\
     &+\int_0^t(\cM^ku+g^k,\varphi)\,dW^k
     +\int_0^t(\cN^{\rho}u+h^{\rho},\varphi)\,dB^{\rho},
     \end{align*}
holds for $t\in[0,T]$ and $\varphi\in C^{\infty}_0(\bR^d)$,
where $\fra^{ij}_j=D_j\fra^{ij}$.
\end{defn}

To present those existence and uniqueness theorems from the $L_2$-theory
of stochastic PDEs which we use in this paper, we formulate some assumptions.

 \begin{assumption}                                                                                             \label{assumption 2.5.1.12}
 There is a constant $\lambda\geq0$ such that for all $n\geq1$,
 $dP\times dt\times dx$ almost all $(\omega,t,x)\in\Omega\times H_T$
we have
 \begin{equation*}
 (\fra^{ij}-\tfrac{1}{2}\frb^{ik}\frb^{jk})z^iz^j\geq \lambda|z|^2
 \quad\text{for all $z=(z^1,...,z^d)\in\bR^d$}.
 \end{equation*}
 \end{assumption}

 To formulate some further conditions
 on the smoothness of the coefficients and the data of
 \eqref{7.22.1}-\eqref{2.5.1.12} we fix
 an integer $m\geq1$.
 We consider first the case $\lambda>0$
 in Assumption \ref{assumption 2.5.1.12},
and make the following conditions.

\begin{assumption}                                                                                         \label{assumption 7.22.1}
The coefficients $\fra^{ij}$, $\fra^i$, $\fra$, $\frb^{ik}$, $\frb^k$,
$\frc^{i\rho}$, $\frc^{\rho}$ and their derivatives in $x\in\bR^d$
up to order $m$ are
$\cP\times\cB(\bR^d)$-measurable real functions on
$\Omega\times H_T$ and in magnitude are bounded  by $K$.
\end{assumption}

\begin{assumption}                                                                                           \label{assumption 7.22.3}
The initial value $u_0$ is an $H^m$-valued random
variable. The free terms $f=f(t)$, $g^k=g^k(t)$, $h^{\rho}=h^{\rho}(t)$
are predictable $H^m$-valued processes such that almost surely
$$
\int_0^T|f(t)|_{m-1}^2\,dt<\infty, \quad \int_0^T|g(t)|^2_{m}\,dt<\infty,\quad
\int_0^T|h^{\rho}(t)|_{m}\,d\|B^{\rho}\|(t)<\infty
$$
for all $k=1,...,d_1$ and $\rho=1,...,d_2$, where $|g|^2_l=\sum_{k}|g^k|^2_l$
and $|h|^2_l=\sum_{\rho}|h^{\rho}|^2_l$ for ever $l\geq0$.
\end{assumption}

\begin{thm}                                                                                            \label{theorem 7.22.1}
Let Assumptions \ref{assumption 2.5.1.12} with $\lambda>0$,
\ref{assumption 7.22.1} and  \ref{assumption 7.22.3} hold.
Then \eqref{7.22.1}-\eqref{2.5.1.12} has a unique generalised solution $u$.
Moreover, $u$ is an $H^m$-valued weakly continuous process, it is
strongly continuous as an $H^{m-1}$-valued process,
$u(t)\in H^{m+1}$ for $P\times dt$ a.e. $(\omega,t)$,  and there exist
constants $\nu\geq0$ and $C>0$ such that for every $l\in[0,m]$
$$
E\sup_{t\leq T}e^{-\nu V}|u|_l^2
+E\int_0^Te^{-\nu V}|u|_{l+1}^2\,dt
$$
$$
\leq
C\left\{|u_0|^2_{l}+\int_0^Te^{-\nu V}(|f|_{l-1}^2+|g|_{l}^2)\,dt
+\int_0^Te^{-\nu V}|h^{\rho}|^2_l\,d\|B^{\rho}\|\right\},
$$
where $V(t)=t+\sum_{\rho=1}^{d_2}\|B^{\rho}(t)\|$. The constants
$\nu$ and $C$ depend only on $\lambda$, $K$, $d$, $d_1$, $d_2$ and $l$.
\end{thm}

In the degenerate case, i.e.,
when $\lambda=0$ in Assumption \ref{assumption 2.5.1.12},
we need to impose somewhat stronger conditions in the other
assumptions of the previous theorem.

\begin{thm}                                                                                                       \label{theorem 1.6.1.12}
Let Assumptions \ref{assumption 2.5.1.12} (with $\lambda=0$),
\ref{assumption 7.22.1} and  \ref{assumption 7.22.3} hold.
Assume, moreover, that the derivatives in $x\in\bR^d$ of $a^{ij}$
up to order $m\vee2$ are bounded real functions on
$\Omega\times[0,T]\times\bR^d$ for all $i,j=1,...,d$, and that
$g^k=g^k(t)$ are $H^{m+1}$-valued predictable processes
for all $k=1,...,d_1$, such that almost surely
$$
\int_0^T|g(t)|_{m+1}^2\,dt<\infty.
<\infty.
$$
 Then the conclusion of Theorem \ref{theorem 7.22.1} remains valid.
\end{thm}

Theorem \ref{theorem 1.6.1.12} is a slight modification of \cite[Theorem~3.1]{Gyongy-Krylov} and can be proved in the same way.
We can prove Theorem \ref{theorem 7.22.1} in the same fashion.

\subsection{Inequalities in Sobolev spaces and a Gronwall-type lemma}

In the following lemmas we present some estimate we
use in the paper. We consider the differential operators
$$
\cM=b^iD_i+b^0,\qquad \cN=c^iD_i+c^0,\qquad \cK=d^iD_i+d^0,
$$
and
$$
\cL=a^{ij}D_iD_j+a^iD_i+a^0,
$$
where
$a^{ij}$, $a^i$, $a^0$, $b^i$, $b^0$, $c^i$, $c^0$, $d^i$ and $d^0$ are
Borel functions defined on $\Real^d$ for $i,j=1,\ldots,d$.
We fix an integer $l\ge0$ and
a constant $K$.
Recall the notation $(\cdot\,,\cdot)=(\cdot\,,\cdot)_0$ for the inner
product in $H^0\equiv L^2(\Real^n)$, and $\<\cdot\,,\cdot\>$ for the duality product
between $H^1$ and $H^{-1}$.

\begin{lem}                                                                                       \label{Lemma:Gyongy}
(i) Assume that $b^0$ and its derivatives up to order $l$, and $b^i$
  and their derivatives up to order $l\vee1$ are real functions,
  in magnitude bounded by $K$.  Then for a constant $C=C(K,l,d)$
  \begin{align}                                                                         \label{3.6.1.12}
  |(D^{\alpha}\cM v,D^{\alpha}v)|&\leq C|v|_l^2,
      \\
   |(D^{\alpha}\cM v,D^{\alpha}u)
   +(D^{\alpha}\cM u,D^{\alpha}v)|
   &\leq  C|v|_l |u|_l
   \end{align}

      for all $u,v\in H^{l+1}$ and multi-indices $\alpha$, $|\alpha|\leq l$.

 (ii) Assume that $b^0$, $c^0$ and their derivatives up to order $l\vee1$,
 $b^{i}$ and $c^{i}$ and their derivatives up to order $(l+1)\vee2$ are
 real functions, in magnitude bounded by $K$.
 Then for a constant $C=C(K,l,d)$
$$
|(D^{\alpha}\cM\cN v,D^{\alpha}v)
+(D^{\alpha}\cM v,D^{\alpha}\cN v)|
\leq C|v|_l^2,
$$
for all $v\in H^{l+2}$ and multi-indices $\alpha$, $|\alpha|\leq l$.

(iii) Assume that $b^0$, $c^0$, $d^0$
and their derivatives up to order $(l+1)\vee2$, $b^i$, $c^i$, $d^i$
and their derivatives up to order $(l+2)\vee3$ are real functions,
and in magnitude are bounded
by $K$ for $i=1,...,d$. Then  for a constant $C=C(K,l,d)$
\begin{align*}
&|(D^{\alpha}\cM\cN\cK v,D^{\alpha}v)
 +(D^{\alpha}\cM\cN v,D^{\alpha}\cK v)\\
 &+(D^{\alpha}\cM\cK v,D^{\alpha}\cN v)
 +(D^{\alpha}\cM v,D^{\alpha}\cN\cK v)|
  \leq C|v|_l^2,
\end{align*}
for all $v\in H^{l+3}$ and multi-indices $\alpha$, $|\alpha|\leq l$.
\end{lem}

\begin{proof}
These and similar estimates are proved in \cite{G2}.
For the sake of completeness and the convenience of the reader
we present a proof here. We can assume that $v\in C^\infty_0(\Real^n)$.
Let us start with $(i)$. Integrating by parts,
we have
$$
(\cM D^\alpha v,D^\alpha v)
=-(D^\alpha v,\cM D^\alpha v)+(D^\alpha v,\bar{m}D^\alpha v),
$$
where $\bar{m}:=2b^0-\sum_{i=1}^dD_ib^i$.
Therefore, by writing $[\cM,D^\alpha]=D^\alpha\cM-\cM D^\alpha$,
\begin{align*}
    (D^\alpha \cM v,D^\alpha v) &= (\cM D^\alpha v,D^\alpha v)
    +([\cM,D^\alpha]v,D^\alpha v) \\
     &=\tfrac{1}{2}(D^\alpha v,\bar{m}D^\alpha v)
     +([\cM,D^\alpha]v,D^\alpha v)\leq C|v|_l^2,
\end{align*}
by the regularity assumed  on the coefficients.
 Let us write
$$
p(v):=(D^\alpha \cM v,D^\alpha v)
=\tfrac{1}{2}(D^\alpha v,\bar{m}D^\alpha v)
+([\cM,D^\alpha]v,D^\alpha v)=:q(v)+r(v).
$$
Defining
$$
2a(u,v):=p(u+v)-p(u)-p(v)=(D^\alpha \cM u,D^\alpha v)
+(D^\alpha \cM v,D^\alpha u),
$$
$$
2b(u,v):=q(u+v)-q(u)-q(v)=(\bar{m}D^\alpha u,D^\alpha v),
$$
and
$$
2c(u,v):=r(u+v)-r(u)-r(v)
=([\cM,D^\alpha]u,D^\alpha v)+([\cM,D^\alpha]v,D^\alpha u),
$$
we have
\begin{equation}                                                                            \label{6.6.1.12}
a(u,v)=b(u,v)+c(u,v)
\end{equation}
and
$$
|a(u,v)|\leq |b(u,v)|+|c(u,v)|\leq C|u|_l|v|_l,
$$
which proves the second inequality in \eqref{3.6.1.12}.
The identity \eqref{6.6.1.12}
applied with $u=Nv$ establishes that
$$
(D^\alpha \cM\cN v,D^\alpha v)+(D^\alpha \cM v,D^\alpha \cN v)
 = (\bar{m}D^\alpha \cN v,D^\alpha v)
 $$
 $$
 +([\cM,D^\alpha]\cN v,D^\alpha v)+([\cM,D^\alpha]v,D^\alpha\cN v).
$$
By the previous case,
$$
|(D^\alpha \cM\cN v,D^\alpha v)+(D^\alpha \cM v,D^\alpha \cN v)|
\leq C|v|_l^2,
$$
and $(ii)$ is proved.
For $(iii)$, integrating by parts,
\begin{align*}
\widetilde{p}(v)
:=&(D^\alpha \cM\cN v,D^\alpha v)+(D^\alpha \cM v,D^\alpha
\cN v) \\
=&(D^\alpha \cN v,\bar{m}D^\alpha v)
+([\cM,D^\alpha \cN]v,D^\alpha v)+([\cM,D^\alpha]v,D^\alpha\cN v)\\
=&\widetilde{q}(v)+\widetilde{r}(v)+\widetilde{s}(v).
\end{align*}
By polarizing this last identity as above and letting $u=\cK v$,
we have
$$
(D^\alpha \cM\cN\cK v,D^\alpha v)
+(D^\alpha\cM\cN v,D^\alpha \cK v)
$$
$$
+(D^\alpha\cM\cK v,D^\alpha\cN v)
    +(D^\alpha \cM v,D^\alpha \cN\cK v)
$$
$$
=(D^\alpha \cN\cK v,\bar{m}D^\alpha v)
     +(D^\alpha \cN v,\bar{m}D^\alpha \cK v)
$$
$$
+([\cM,D^\alpha \cN]\cK v,D^\alpha v)
+([\cM,D^\alpha \cN]v,D^\alpha \cK v)
$$
$$
([\cM,D^\alpha ]\cK v,D^\alpha\cN v)
     +([\cM,D^\alpha]v,D^\alpha \cN\cK v)\leq  C|v|_l^2,
$$
where in the last inequality we used $(ii)$. Hence $(iii)$ is proved.
\end{proof}

\begin{lem}                                                                                  \label{lemma 1.10.1.12}
Assume that $a^{ij}$, $b^0$ and their derivatives
up to order $l\vee1$, $b^i$ and their derivatives up to order
$l\vee2$, $a^i$, $a^0$ and their derivatives up to order $l$
are real functions, in magnitude bounded by $K$ for $i=1,...,d$.
Then  for a constant $C=C(K,l,d)$
$$
|(D^{\alpha}\cM v,D^{\alpha}\cL v)+\<D^{\alpha}v,D^{\alpha}\cM\cL v\>|
\leq C|v|_{l+1}^2,
$$
for $v\in H^{l+2}$ and multi-indices $|\alpha|\leq l$.
\end{lem}

\begin{proof}
Let us check first the case $l=0$.
Denote by $\cM^\ast$ the formal adjoint of $\cM$. We have
$$
|(\cM v,\cL v)+\<v,\cM\cL v\>=|((\cM+\cM^{\ast})v,\cL v)|
=|(-b^i_iv+2b^0v,\cL v)|\leq C |v|_1,
$$
where $b^i_i=D_ib^i$.
For the general case, let $\abs{\alpha}\leq l$ and write
\begin{align*}
    &(D^\alpha \cM v,D^\alpha \cL v)+\<D^\alpha v,D^\alpha \cM\cL v\> \\
     &= (D^\alpha Mv,D^\alpha Lv)
     +\<D^\alpha v,MD^\alpha \cL v\>+(D^\alpha v,[\cM,D^\alpha]\cL v) \\
     &= (D^\alpha Mv,D^\alpha \cL v)
     +(M^\ast D^\alpha v,D^\alpha \cL v)+(D^\alpha v,[\cM,D^\alpha]Lv) \\
     &= ([M,D^\alpha]v,D^\alpha \cL v)
     +(-b^i_iD^\alpha v+2b^0D^\alpha v,D^\alpha Lv)
     +(D^\alpha v,[\cM,D^\alpha]\cL v),
\end{align*}
from which the estimate follows.
\end{proof}

The next lemma is a standard fact for elliptic differential operators
$\cL=a^{ij}D_iD_j+a^iD_i+a^0$
\begin{lem}                                                                                      \label{Lemma:vLv}
Assume there exists a constant $\lambda>0$ such that
$$
a^{ij}(x)z^iz^j\geq\lambda\abs{z}^2,\quad\text{for all $z,x\in\bR^d$},
$$
and that the derivatives of $a^i$ and $a^0$ up to order $(l-1)\vee0$,
and the derivatives of $a^{ij}$ up to order $l\vee1$ are functions,
bounded by $K$, for $i,j=1,...,d$.
Then there is a constant $C=C(K\lambda,l,d)$ such that
$$
(D^{\alpha}v,D^{\alpha}\cL v)\leq C|v|_l^2-\frac{\lambda}{2}|v|_{l+1}^2,
$$
for all $v\in H^{l+2}$ and multi-indices $|\alpha|\leq l$.
\end{lem}

In the next two lemmas we assume that there exist vector fields
$$
\sigma^1=(\sigma^{i1}(x)),\ldots,\sigma^p=(\sigma^{ip}(x)),
$$
 such that $a^{ij}=\sigma^{ir}\sigma^{jr}$ for all $i,j=1,\ldots,d$. Set
\begin{equation*}                                                  \label{Kr}
\cN^r:=\sigma^{ir}D_i,\qquad r=1,\ldots,p,
\end{equation*}
and notice that if the $\sigma^r$ are differentiable then we can write
$\cL=\sum_{r=1}^p(\cN^r)^2+\cN^0$,
where $\cN^0=\left(a^j-\sigma^{ir}(D_i\sigma^{jr})\right)D_j+a^0$.

\begin{lem}                                                          \label{lemma 19.2.1.12}
Assume that the
derivatives of $\sigma$  up to order $(l+1)\vee2$
and the derivatives of $a^i$, $a^0$  up to order $l\vee1$
are functions, bounded by a constant $K$ for $i=1,...,d$.
Then
$$
(D^{\alpha}\cL v,D^{\alpha}v)\leq -\sum_{r=1}^p|D^{\alpha}\cN^rv|^2_0+C|v|^2_l,
$$
and
$$
|(D^{\alpha}\cL v,D^{\alpha}u)|\leq \sum_{r=1}^p|D^{\alpha}\cN^rv|^2_0+C(|v|_l^2+|u|^2_{l+1}),
$$
for all $v,u\in H^{l+2}$ and multi-indices $|\alpha|\leq l$, with a constant $C=C(K,d,l,p)$.
\end{lem}

\begin{proof}
By Lemma \ref{Lemma:Gyongy} $(ii)$ and $(iii)$,
\begin{align*}
(D^{\alpha}v,D^{\alpha}\cL v)
=&(D^{\alpha}v,D^{\alpha}\cN^r\cN^r v)_0+(D^{\alpha}v,D^{\alpha}\cN^0v)_{0}         \\
\leq&
-(D^{\alpha}\cN^rv,D^{\alpha}\cN^rv)_{0}+C|v|_l^2,
\end{align*}
with a constant $C=C(K,d,l,p)$ and the first inequality of the statement follows. To get the second one we need only note
that by interchanging differential operators and by integration by parts we have
\begin{align*}
|(D^{\alpha}\cN^r\cN^rv,D^{\alpha}u)|
\leq& |(D^{\alpha}\cN^rv,\cN^rD^{\alpha}u)|
+|([\cN^r,D^{\alpha}]\cN^rv,D^{\alpha}u)|+C|u|_{l+1}^2                    \\
\leq& \sum_{r=1}^{p}|D^{\alpha}\cN^rv|^2_0+C|u|^2_{l+1},                         \\
|(D^{\alpha}\cN^0v,D^{\alpha}u)|\leq&
C(|v|_l^2+|u|_{l+1}^2),
\end{align*}
with constants $C=C(K,d,p,l)$.
\end{proof}

\begin{lem}                                                                                \label{Lemma:ML}
Assume that the derivatives of $\sigma^i$
and $b^i$ up to order $(l+2)\vee3$
and the derivatives of $a^i$, $a^0$ and $b^0$ up to order $(l+1)\vee2$
are functions, bounded by a constant $K$ for $i=1,...,d$.
Then
$$
|(D^{\alpha}\cM\cL v,D^{\alpha}v)
+(D^{\alpha}\cL v,D^{\alpha}\cM v)|\leq C\sum_{r=1}^p|\cN^rv|^2_l+C
|v|_l^2,
$$
with a constant $C=C(K,l,d,p)$ for all $v\in H^{l+3}$ and multi-indices
$|\alpha|\leq l$.
\end{lem}

\begin{proof}
Put $\cN=\cK=\cN^r$, $r=1,\ldots,p$ in Lemma \ref{Lemma:Gyongy} $(iii)$
and use $(i)$ of the same lemma for $\cN^rv$ to get
$$
|(D^{\alpha}\cM\cN^r\cN^rv,D^{\alpha}v)
+(D^{\alpha}\cM v,D^{\alpha}\cN^r\cN^rv)|
$$
$$
\leq C|v|_l^2+2|(D^{\alpha}\cM\cN^rv,D^{\alpha}\cN^rv)|
\leq C|v|_l^2+C\sum_{r=1}^{p}|\cN^rv|_l^2,
$$
and apply Lemma \ref{Lemma:Gyongy} $(ii)$ to obtain
\begin{equation*}
|(D^{\alpha}\cM\cN^0v,D^{\alpha}v)+(D^{\alpha}\cM v,D^{\alpha}\cN^0v)|
\leq C|v|_l^2,
\end{equation*}
which prove the corollary.
\end{proof}

The following Gronwall type lemma will be useful
for our estimates in the next section.

\begin{lem}                                                                                                    \label{Lemma:gronwall}
Let $y_n$, $m_n$, $Q_n$ and $q_n$
be sequences of real valued continuous $\cF_t$-adapted
stochastic processes given on the interval $[0,T]$,
such that $Q_n$ is a non-decreasing non-negative process
and $m_n$ is a local martingale starting from $0$.
Let $\delta$, $\gamma$ be some real numbers with $\delta<\gamma$.
Assume that almost surely
\begin{equation}\label{pregronwall}
0\leq y_n(t)\leq \int_0^t y_n(s) \,dQ_n(s)+m_n(t)+q_n(t),
\end{equation}
holds for all $t\in[0,\tau_n]$ and integers $n\geq1$, where
$$
\tau_n=\inf\set{t\geq 0:y_n(t)\geq n^{-\delta}}\wedge T.
$$
Suppose that almost surely
$$
Q_n(\tau_n)=o(\ln n),\quad\sup_{t\leq\tau_n}q_n(t)=O(n^{-\gamma}),
$$
$$
d\<m_n\>\leq (y_{n}^2+k_ny_n)\,dQ_n,
~\hbox{on}~t\in[0,\tau_n],\quad\int_0^{\tau_n}k_n(s)\,dQ_n=O(n^{-\gamma})
$$
for a sequence of non-negative ${\mathcal F}_t$-adapted processes $k_n$.
Then almost surely
\begin{equation}                                                                                                          \label{conclusion}
\sup_{t\leq T}y_n(t)=O(n^{-\kappa}),\quad\hbox{for each}~\kappa<\gamma.
\end{equation}
\end{lem}

\begin{proof}
Let us assume first that $\gamma>0$. The case $\delta=0$ is a
slight modification of  \cite[Lemma~3.8]{Gyongy-Shmatkov}.
It can be proved in the same way
by using a suitable generalization of Lemma 3.7 from \cite{Gyongy-Shmatkov}
(see \cite{konyv}).
For $\delta<\gamma\in(0,\infty)$, we see that the conditions
of the Lemma are satisfied with $\gamma'=\gamma-\delta$,
\begin{equation}\label{primes}
y_n'(t)=\frac{y_n(t)}{n^{-\delta}},\quad m_n'(t)
=\frac{m_n(t)}{n^{-\delta}},
\quad q_n'(t)=\frac{q_n(t)}{n^{-\delta}},
\quad k_n'(t)=\frac{k_n(t)}{n^{-\delta}},
\end{equation}
in place of $\gamma$, $y_n$, $m_n$, $q_n$ and $k_n$,
with $\delta=0$.
Hence we have \eqref{conclusion} for $y_n'$ in place of $y_n$
for each $\kappa<\gamma'$, which gives
\eqref{conclusion} in this case.

Suppose now that $\gamma\leq0$.
Take $\bar\gamma\in(\delta,\gamma)$
and set $\gamma':=\gamma-\bar\gamma$,
$\delta':=\delta-\bar\gamma$.
Define $y_n'$, $m_n'$, $q_n'$ and $k_n'$ as
in \eqref{primes} with $\bar\gamma$ in place of $\delta$. Notice
that $\gamma'>0$ and that the conditions
of the lemma are satisfied by the processes $y_n'$, $m_n'$, $q_n'$
and $k_n'$ in place of $y_n$, $m_n$, $q_n$ and $k_n$,
with $\gamma'$ and $\delta'$ instead of $\gamma$ and
$\delta$. Hence \eqref{conclusion} holds for $y_n'$
and $\gamma'$ in place of $y_n$ and $\gamma$, which gives
the lemma.
\end{proof}

\begin{cor}                                                                                            \label{Corollary:Gronwall stopped}
Let $\sigma_n$ be an increasing sequence of stopping times
converging to infinity almost surely.
Assume that the conditions of the previous lemma are satisfied with
$\bar\tau_n=\inf\{t\geq0:y_n(t)\geq n^{-\delta}\}\wedge T\wedge\sigma_n$
 in place of $\tau_n$.
Then its conclusion, \eqref{conclusion}, still holds.
\end{cor}

\begin{proof}
The conditions of Lemma \ref{Lemma:gronwall}
are satisfied by the processes
$$
y_n'(t)=y_n(t\wedge\sigma_n),\quad m_n'(t)=m_n(t\wedge\sigma_n),
\quad q_n'(t)=q_n(t\wedge\sigma_n),
$$
$$
k_n'(t)=k_n(t\wedge\sigma_n),
$$
in place of $y_n$, $m_n$, $q_n$ and $k_n$
and with  $\tau_n'=\inf\set{t\geq0:y_n'(t)\geq n^{-\delta}}$ in
place of $\tau_n$. Hence
\begin{equation}\label{ynprime}
\sup_{t\leq T}y_n'(t)=O(n^{-\kappa}),\quad\hbox{for each}~\kappa<\gamma.
\end{equation}
Define the set $\Omega_n:=[\sigma_n\geq T]$ and note
that since $\sigma_n\nearrow\infty$ almost surely,
the set $\Omega'=\cup_{n\geq1}\Omega_n$ has full
probability. It remains to prove that the random variable
$$\xi:=\sup_{m\geq1}\sup_{t\leq T}\frac{y_m(t)}{m^{-\kappa}}$$
is finite almost surely for all $\kappa<\gamma$.
Indeed, take $\omega\in\Omega'$.
Then $\omega\in\Omega_n$ for some $n(\omega)\geq1$, hence
$\sigma_m(\omega)\geq T$ for all $m\geq n(\omega)$ and, by \eqref{ynprime},
$$\sup_{t\leq T}y_m(t\wedge\sigma_m)=\sup_{t\leq T}y_m(t)\leq\zeta_\kappa m^{-\kappa},$$
for all $m\geq n(\omega)$. Since $\zeta_\kappa$ is finite almost surely, so it is $\xi$.
\end{proof}

\section{The growth of the approximations}                           \label{section growth}

In this section we estimate solutions $u_n$ of \eqref{eq for un} for large
$n$. We fix an integer $l\geq0$,
a constant $K\geq0$,
and make the following assumptions.

\begin{assumption}                                                               \label{assumption coefnondeg}
The derivatives in $x$ of the coefficients $a^{ij}_n$, $a^i_n$,
$a_n$, $b^k_n$ up to order $l+1$, and the derivatives in $x$ of $b^{ik}_n$
up to order $(l+1)\vee2$ are $\cP\times\cB(\bR^d)$-measurable real functions
on $\Omega\times[0,T]\times\bR^d$
and in magnitude are bounded by $K$, for all $i,j=1,...,d$, $k=1,...,d_1$,
and all $n\geq1$.
\end{assumption}

\begin{assumption}                                                                    \label{assumption f and g}
For each $\varepsilon>0$
almost surely
$$
|u_{n0}|_{l}=O(n^{\varepsilon}),
\quad \int_0^T\left(|f_n|^2_{l}+|g_n|^2_{l+1}\right)\,dt=O(n^{\varepsilon}),
\quad
\sup_{t\leq T}|g_{n}(t)|_{l}^2=O(n^{\varepsilon}).
$$
\end{assumption}

We will often use the notation
notation $f\cdot V(t)$ for the integral
$$
\int_0^tf(s)\,dV(s),
$$
when $V$ is a semimartingale
and $f$ is a predictable process such that the stochastic {integral of $f$ against
$dV$ over $[0,t]$ is well-defined. We define
$$
\eta_n(t)=\sup_{m\geq n}\sup_{s\leq t}\abs{W(s)-W_m(s)}.
$$
Notice that Assumption \ref{assumption 15.04.06} (i) clearly implies
that $\eta_n(T)=O(n^{-\kappa})$ almost surely for each $\kappa<\gamma$.

First we study the case when
$b^{ik}_n$, $b^k_n$ and $g^{k}_n$ do not depend on $t\in[0,T]$.

\begin{thm}                                                                     \label{theorem growth nondeg}
Assume that
$b^{ik}_n$, $b^k_n$ and $g^{k}_n$
do not depend on $t$ for
$i=1,...,d$, $k=1,...,d_1$ and $n\geq1$.
Let Assumptions \ref{assumption 15.04.06} $\mathrm{(i)}$ and $\mathrm{(iii)}$,
\ref{assumption coefnondeg},
\ref{assumption f and g} and
\ref{assumption 06.02.06} with $\lambda>0$ hold.
Then for every $\varepsilon>0$ almost surely
$$
\sup_{t\leq T}\abs{u_n(t)}_l^2+\int_0^T|u_n|^2_{l+1}\,dt=O(n^\varepsilon).
$$
\end{thm}

\begin{proof}
Assume for the moment that $u_{n0}\in H^{l+1}$ almost surely. Recall that we are assuming that $W_n^k$ is of bounded variation, $n\geq1$, $k=1,\ldots,d_1$.
Then by Theorem \ref{theorem 7.22.1},
under Assumptions \ref{assumption coefnondeg},
\ref{assumption f and g} and  \ref{assumption 06.02.06} with $\lambda>0$ there
is a unique generalised solution  $u_n$ of \eqref{eq for un}-\eqref{ini}, and it is
an $H^{l+1}$-valued  weakly continuous
process such that almost surely
$$
\int_0^T|u_n(t)|_{l+2}^2\,dt<\infty.
$$
 In particular,
$$
(u_n(t),\varphi)_0=(u_{n0},\varphi)_0
$$
$$
+\int_0^t\big[-(a^{ij}_nD_iu_n,D_j\varphi)_0
+\big((a^i_n-D_ja_{n}^{ij})D_iu_n+a_nu_n+f_n,\varphi\big)_0\big]\,ds
$$
$$
+\int_0^t(b^{ik}_nD_iu_n+b^k_nu_n+g^k_n,\varphi)_0\,dW_n^k
$$
holds for all $t\in[0,T]$ and $\varphi\in C^{\infty}_0(\bR^d)$.
Substituting here $\sum_{|\alpha|\leq l}(-1)^{|\alpha|}D^{2\alpha}\varphi$
in place of $\varphi$ and integrating by parts, we
get
$$
(u_n(t),\varphi)_l=(u_{n0},\varphi)_l
$$
$$
+\int_0^t\big[-(a^{ij}_nD_iu_n,D_j\varphi)_l
+\big((a^i_n-D_ja_n^{ij})D_iu_n+a_n u_n+f_n,\varphi\big)_l\big]\,ds
$$
$$                                \\
+\int_0^t(b^{ik}_nD_iu+b^k_nu_n+g^k_n,\varphi)_l\,dW_n^k(s).
$$
Hence using It\^o's formula in the triple
$H^{m+1}\hookrightarrow  H^m\equiv(H^m)^{\ast}\hookrightarrow H^{m-1}$,
we have
$$
d|u_n|^2_l=2(u_n,L_nu_n+f_n)_l\,dt
+2(u_n,M^k_nu_n+g^k_n)_l\,dW_n^k =2(u_n,L_nu_n+f_n)_l\,dt
$$
$$
+2(u_n,M^k_nu_n+g^k_n)_l\,dW^k
     +2(u_n,M^k_nu_n+g^k_n)_l\,d(W_n^k-W^k).
$$
Hence, integrating by parts in the last term above we have
\begin{equation}                                                                               \label{6.10.1.12}
\abs{u_n}_l^2=\sum_{j=1}^6I_n^{(j)},
\end{equation}
where
\begin{align*}
    I^{(1)}_n &= \abs{u_{n0}}_l^2, \\
    I^{(2)}_n &= 2(u_n,L_nu_n+f_n)_l\cdot t, \\
    I^{(3)}_n &= 2(u_n,M^k_nu_n+g^k_n)_l\cdot W^k(t), \\
    I^{(4)}_n &= 2(u_n,M^k_nu_n+g^k_n)_l(W_n^k-W^k)\Big|_0^t, \\
    I^{(5)}_n &= 2\left(\{(L_nu_n+f_n,M^k_nu_n+g^k_n)_l
    +\<u_n,M^k_n(L_nu_n+f_n)\>_l\}(W^k-W_n^k)\right)\cdot t, \\
    I^{(6)}_n &= 2\left((M^j_nu_n+g^j,M^k_nu_n+g^k_n)_l
    +\<u_n,M^k_n(M^j_nu_n+g^j_n)\>_l\right)\cdot B_n^{kj}(t),
\end{align*}
and
$\<\cdot\,,\cdot\>_l$ is the duality product
between $H^{l+1}$ and $H^{l-1}$,
based
on the inner product $(\cdot\,,\cdot)_l$ in $H^l$.
Using Lemma \ref{Lemma:vLv} we obtain
$$
I^{(2)}_n\leq\left[C\abs{u_n}_l^2-\lambda|u_n|_{l+1}^2
+C|f_n|_{l-1}^2\right]\cdot t
$$
with a constant $C=C(K,l,d,\lambda)$.
The term $I^{(3)}_n$ is a continuous local martingale starting from $0$,
such that its quadratic variation,  $\<I^{(3)}_n\>$, satisfies, by Lemma \ref{Lemma:Gyongy} \textit{(i)}
\begin{equation}                                                                                  \label{1.10.1.12}
d\<I^{(3)}_n\>\leq C(|u_n|_l^4+|u_n|_l^2|g_n|_l^2)\,dt,
\end{equation}
and also by Lemma \ref{Lemma:Gyongy} \textit{(i)}
we have
$$
I^{(4)}_n(t)\leq C\sup_{0\leq s\leq t}\left(\abs{u_n(s)}_l^2
+|g_n(s)|_l^2\right)\eta_n(t)
$$
with a constant $C=C(K,d,d_1,l)$,
where $|g_n|_l^2=\sum_k|g^k_n|_l^2$.
Using Lemmas
\ref{Lemma:Gyongy} \textit{(ii)} and \ref{lemma 1.10.1.12},
we have
$$
I^{(5)}_n(t)\leq C\left[\abs{u_n}_{l+1}^2+|f_n|_l^2
+\abs{g_n}_{l+1}^2\right]\eta_n\cdot t,
$$
and
$$
I^{(6)}_n(t)\leq C\left[\abs{u_n}_l^2+|g_n|_{l}^2\right]\cdot \|B_n\|(t),
$$
with $\|B_n\|:=\sum_{j,k}\|B^{jk}_n\|$
and a constant $C=C(K,d,d_1,l)$.
Therefore, from \eqref{6.10.1.12},
\begin{multline}                                                                             \label{3.10.1.12}
    |u_n(t)|_l^2
    \leq \int_0^t(|u_n(s)|_l^2+|g_n|_{l}^2)\,dQ_n(s) +m_n(t)                                   \\
-\int_0^t\left(\lambda-C\eta_n(s)\right)|u_n(s)|_{l+1}^2\,ds
    +C\int_0^t(|f_n(s)|_l+|g_n|_{l+1}^2)(1+\eta_n(s))\,ds                                       \\
+|u_0|_l^2+C\sup_{0\leq s\leq t}
\left(\abs{u_n(s)}_l^2+|g_n(s)|_l^2\right)\eta_n(t),
\end{multline}
with $m_n=I^{(3)}_n$ and $Q_n(s)=C(s+\|B_n\|(s))$. Define
$$
\sigma_n:=\inf\set{t\geq0:2C\eta_n(t)\geq{\lambda}},
$$
and note that almost surely
\begin{align}
    y_n(t) &:=|u_n(t)|_l^2+ \frac{\lambda}{2}\int_0^t|u_n(s)|_{l+1}^2\,ds
    \leq\int_0^t|u_n(s)|_l^2\,dQ_n(s)+m_n(t)+q_n(t)                                          \nonumber\\
     &\leq \int_0^ty_n(s)\,dQ_n(s)+m_n(t)+q_n(t)
     \quad \text{for all $t\in[0,\sigma_n]$},                                            \label{1.3.3.12}
\end{align}
with $q_n=q^{(1)}_n+q^{(2)}_n$,
where
\begin{align*}
q^{(1)}_n(t)=&|u_{n0}|_l^2
+C\eta_n(t)\sup_{0\leq s\leq t}|g_n|_l^2                                                                           \\
&+C\int_0^t(1+\eta_n)(|f_n(s)|_l^2+|g_n|_{l+1}^2)\,ds
+C\int_0^t|g_n|_{l}^2\,d\|B_n\|(s),
\\ q_n^{(2)}(t)=&C\eta_n(t)\sup_{0\leq s\leq t}|u_n(s)|_l^2.
\end{align*}
Due to Assumptions \ref{assumption 15.04.06} $\mathrm{(i)}\&\mathrm{(iii)}$
and
\ref{assumption f and g} we have
$$
\sup_{t\leq T}|q^{(1)}_n(t)|=O(n^{\varepsilon})
\quad \text{almost surely for each $\varepsilon>0$}.
$$
For a given $\kappa\in(0,\gamma)$ and any $\varepsilon\in (0,\kappa)$
take $\bar{\varepsilon}\in(\varepsilon,\kappa)$ and define
$$
\tau_n=\inf\{t\geq0: |u_n(t)|_l^2\geq n^{\bar{\varepsilon}}\}.
$$
Then clearly
$$
\sup_{t\leq \tau_n}|q^{(2)}_n(t)|=O(n^{\varepsilon})
\quad \text{a.s. for each $\varepsilon>0$}.
$$
Thus
$$
\sup_{t\leq \tau_n}|q_n(t)|=O(n^{\varepsilon})
\quad \text{a.s. for each $\varepsilon>0$}.
$$
Now taking into account \eqref{1.10.1.12} and
noting that $\sigma_n\nearrow\infty$, we finish
the proof of the theorem by applying
Corollary \ref{Corollary:Gronwall stopped}
to \eqref{1.3.3.12}.
Since our estimates do not depend on the norm
of $u_{n0}$ in $H^{l+1}$ but on its norm in $H^l$,
by a standard approximation argument we can relax the assumption
that $u_{n0}$ is almost surely in $H^{l+1}$.
\end{proof}

\smallskip
In the case when $b^{ik}_n$, $b^k_n$ and $g^k_n$ depend
on $t$ we make the following assumption.

\begin{assumption}                                                        \label{assumption 5.10.1.12}
For each $n\geq1$, $i=1,2,...,d$ and $k=1,...,d_1$
there exist real-valued $\cP\times\cB(\bR^d)$-measurable
functions $b^{ik(r)}_n$, $b^{k(r)}_n$ on
$\Omega\times[0,T]\times\bR^d$ and
$H^0$-valued predictable processes $g^{k(r)}_n$
for $r=0,1,...,d_1$,
such that
almost surely
$$
d(b^{ik}_n(t),\varphi)=(b^{ik(0)}_n(t),\varphi)\,dt
+(b^{ik(p)}_n(t),\varphi)\,dW^{p}_n(t),
$$
$$
d(b^{k}_n(t),\varphi)=(b^{k(0)}_n(t),\varphi)\,dt
+(b^{k(p)}_n(t),\varphi)\,dW^{p}_n(t),
$$
$$
d(g^k_n(t),\varphi)=(g^{k(0)}_n(t),\varphi)\,dt
+(g^{k(p)}_n(t),\varphi)\,dW^{p}_n(t),
$$
for all $i=1,....,d$, $k=1,...,d_1$, every $n\geq1$
and $\varphi\in C^{\infty}_0(\bR^d)$.
The functions $b^{ik(r)}_n$ together with their derivatives in $x$
up to order $l\vee1$ and  the functions $b^{k(r)}_n$ together
with their derivatives in $x$ up to order $l$ are
$\cP\times\cB(\bR^d)$-measurable functions, bounded by $K$ for
all $n\geq1$ and $r=0,1,...,d_1$.
Moreover, for each $\varepsilon>0$
$$
\int_0^T|g^{k(0)}_n|_{l-1}^2\,dt=O(n^{\varepsilon}),
\quad
\sup_{t\leq T}\sum_{p=1}^{d}|g^{k(p)}_n|_{l}^2
=O(n^{\varepsilon})\quad\text{for $k=1,...,d_1$.}
$$
\end{assumption}

\begin{thm}                                                       \label{theorem 1.10.1.12}
Let the assumptions of Theorem \ref{theorem growth nondeg}
together with Assumption \ref{assumption 5.10.1.12} hold.
Then we have the conclusion of Theorem \ref{theorem growth nondeg}.
\end{thm}

\begin{proof}
We can follow the proof of the previous theorem with minor changes.
We need only add an additional term,
$$
I^{(7)}_n=2\{(W^k-W_n^k)(u_n, M^{k(0)}_nu_n+g^{k(0)})_l\}\cdot t
+2(u_n, M^{k(p)}_nu_n+g^{k(p)})_l\cdot B^{kp}_n(t),
$$
to the right-hand side \eqref{6.10.1.12},
where for each $n\geq1$,
$$
M^{k(r)}_n=b^{ik(r)}_nD_i+b^{k(r)}_n,
$$
for $k=1,...,d_1$ and $r=0,..,d_1$. Clearly,
$2(u_n, g^{k(0)})_l\leq |u_n|^2_{l+1}+|g^{k}|^2_{l-1}$, and hence
$$
|2(W^k-W_n^k)(u_n, M^{k(0)}_nu_n+g^{k(0)})_l|         \\
$$
$$
\leq
\eta_n(|u_n|_{l+1}^2+|g^{(0)}_n|^2_{l-1}+C|u_n|^2_l),                                          \\
$$
and, by Lemma \ref{Lemma:Gyongy} \textit{(i)},
$$
2|(u_n, M^{k(p)}_nu_n+g^{k(p)})_l|
\leq C|u_n|^2_l+|g^{k(p)}|^2_l
$$
with a constant $C=C(K,d,d_1,l)$,
where $|g^{(0)}_n|^2_{l-1}=\sum_{k=1}^{d_1} |g^{k(0)}_n|^2_{l-1}$.
Thus  inequality \eqref{3.10.1.12}
holds with the additional term
$$
q^{(3)}_n(t)=\eta_n(t)\int_0^t|g^{(0)}_n(s)|^2_{l-1} \,ds
+\int_0^t |g^{k(p)}(s)|^2_l\,d\|B^{kp}\|(s)
$$
added to its left-hand side and
with a constant $C=C(\lambda,K,d_1,l)$.
Since due to Assumptions \ref{assumption 5.10.1.12}
and \ref{assumption 15.04.06} (iii),
for each $\varepsilon>0$ we have
$$
\sup_{t\leq T}q^{(3)}_n(t)=O(n^{\varepsilon})\quad \text{almost surely},
$$
we can finish the proof as in the proof of Theorem \ref{theorem growth nondeg}.
\end{proof}

Let us consider now the degenerate case,
$\lambda=0$ in Assumption \ref{assumption 06.02.06}.

\begin{assumption}                                                                               \label{assumption coefdeg}
For each $n\geq1$ there exist real-valued functions $\sigma_n^{ip}$
on $\Omega\times H_T$ for $p=1,2,...,d_2$
such that $a^{ij}_n=\sigma_n^{ip}\sigma_n^{jp}$
for all $i,j=1,\ldots,d$.
For all $n\geq1$ the functions  $\sigma_n^{ip}$
and $b^i_n$ and their derivatives in $x\in\bR^d$
up to order $(l+2)\vee3$, the functions $a^i_n$, $a_n$, $b_n$ and their
derivatives in $x$ up to order $(l+1)\vee2$ are $\cP\times\cB(\bR^d)$-measurable
functions, bounded by $K$, for all $i=1,...,d$ and $p=1,...,d_2$.
\end{assumption}

\begin{assumption}                                                      \label{assumption 1.19.1.12}
Let Assumption \ref{assumption
5.10.1.12} hold  and assume that for each $\varepsilon>0$
$$
\int_0^T|g_n^{k(0)}(t)|_l^2\,dt=O(n^{\varepsilon})
\quad
\text{almost surely for each $k=1,2,...,d_1$. }
$$
\end{assumption}

\begin{thm}                                                            \label{theorem growth deg}
Let Assumptions \ref{assumption 15.04.06}
$\mathrm{(i)}\&\mathrm{(iii)}$,
 \ref{assumption 06.02.06} (with $\lambda=0$),
 \ref{assumption f and g},
\ref{assumption coefdeg} and \ref{assumption 1.19.1.12}
hold.
Then
$$
\sup_{t\leq T}\abs{u_n(t)}_l^2
+\sum_{r=1}^{d_2}\int_0^T\abs{N^r_nu_n}_l^2\,ds=O(n^\varepsilon),
\quad\text{for every $\varepsilon>0$ almost surely},
$$
where $N^r_n=\sigma^{ir}_nD_i$,  for $r=1,\ldots,d_2$.
\end{thm}

\begin{proof}
The proof follows the lines of that of Theorems \ref{theorem growth nondeg}
and
\ref{theorem 1.10.1.12},
but instead of estimating $I^{(2)}_n$ and $I^{(5)}_n$
separately, we estimate their sum as follows. Note that
\begin{align*}
   d I^{(5)}_n
    = &2\{(L_nu_n+f_n,M^k_nu_n+g^k_n)_l+\<u_n,M^k_n(L_nu_n+f_n)\>_l\}
    (W^k-W_n^k)\,dt \\
     =& 2\{(L_nu_n,M^k_nu_n)_l+\<u_n,M^k_nL_nu_n\>_l
     +(L_nu_n,g^k_n)_l\} (W^k-W_n^k)\,dt \\
     &+2\{(f_n,M^k_nu_n)_l +(u_n,M_n^kf_n)_l\}(W^k-W_n^k)\,dt \\
     &+2(f_n,g_n^k)_l (W^k-W_n^k)\,dt. \\
    \end{align*}
 Using Lemmas \ref{Lemma:ML}, \ref{lemma 19.2.1.12},
 \ref{Lemma:Gyongy} \textit{(i)} and \textit{(ii)}, we obtain
$$
 |(L_nu_n,M^k_nu_n)_l+\<u_n,M^k_nL_nu_n\>_l|
 \leq C\sum_{r=1}^{d_2}|N^r_nu_n|_l^2
 +C|u_n|^2_l,
 $$
$$
(u_n,L_nu_n+f_n)_l
\leq -\sum_{r=1}^{d_2}|N^r_nu_n|_l^2+C|u_n|^2_l+|f_n|^2_l,
$$
$$
|(f_n,M^k_nu_n)_l+(u_n,M^k_nf_n)_l|
\leq C\left(|u_n|_l^2+|f_n|_l^2\right),
$$
$$
|(L_nu_n,g^k_n)_l|\leq \sum_{r=1}^{d_2}|N^r_nu_n|^2+C(|u_n|^2_l+|g^k_n|^2_{l+1})
$$
with a constant $C=C(K,l,d,d_2)$. Hence
$$
 d I^{(5)}_n(t)
  \leq C\eta_n(t)\left(\sum_{r=1}^{d_2}\abs{N^r_nu_n}_l^2
     +|g_n|_{l+1}^2+|f_n|_l^2+|u_n|_l^2\right)\,dt,
$$
and recalling that $I^{(2)}_n(t)=2(u_n,L_nu_n+f_n)_l\cdot t$, we get
\begin{align*}
I^{(2)}_n(t)+I^{(5)}_n(t)
\leq& C\set{(\eta_n+1)(|u_n|_l^2+|f_n|_l^2)
+\eta_n\left|g_n|_{l+1}^2\right)}\cdot t                                          \\
&+\{(C\eta_n-2)\sum_{r=1}^{d_2}|N^r_nu_n|_l^2\}\cdot t.
\end{align*}
To estimate $I^{(7)}_n$ we use that, by Lemma \ref{Lemma:Gyongy}\textit{(i)},
$$
|2(W^k-W_n^k)\{(u_n, M^{k(0)}_nu_n+g^{k(0)}_n)_l\}|
\leq C\eta_n(|u_n|^2_l +|g^{(0)}_n|^2_l)                                         \\
$$
with a constant $C$ and $|g_n^{(0)}|^2_l=\sum_k|g_n^{k(0)}|^2_l$.
Thus using the estimates for $I^{(1)}_n$,
$I^{(3)}_n$,
$I^{(4)}_n$  and $I^{(6)}_n$ given in the proof of
Theorem \ref{theorem growth nondeg}, and defining
$$
\sigma_n=\inf\set{t\geq0: C\eta_n\geq1},
$$
and
$$
y_n(t)=\abs{u_n(t)}_l^2+\sum_{r=1}^{d_2}\int_0^t\abs{K_ru_n}_l^2\,ds,
$$
we get
$$
 y_n(t)\leq \int_0^ty_n(s)\,dQ_n(s)+m_n(t)+q_n(t)
 \quad
 \text{almost surely for all $t\in[0,\sigma_n]$},
$$
with
$$
Q_n(s)=C\{(\eta_n+1)s+\|B_n\|(s)+\eta_n\},
\quad
m_n=I^{(3)}_n,
\quad
q_n=q^{(1)}_n+q^{(2)}_n+\bar q^{(3)}_n,
$$
$$
\bar q_n^{(3)}:=\int_0^t|g^{(0)}_n(s)|^2_{l}\sup_{r\leq t}|W(r)-W_n(r)| \,ds
+\int_0^t |g^{k(p)}_n(s)|^2_l\,d\|B^{kp}_n\|(s),
$$
where $\|B_n\|(s)=\sum_{k,p}\|B^{kp}_n\|(s)$, $q^{(1)}_n$ and $q^{(2)}_n$
are defined in the proof Theorem \ref{theorem growth nondeg}, and $C$ is a constant
depending only on $K$, $d$, $d_1$, $d_2$ and $l$. Hence the proof
is the same as that of Theorem \ref{theorem growth nondeg}.

\end{proof}

\section{Rate of convergence results for SPDEs}                                          \label{section rate}

Here we present two theorems on rate of convergence which provide us with
a technical tool to prove our main results.
Consider  for each integer $n\geq1$ the problem
\begin{align}
du_n(t,x)=&(\cL_nu_n(t,x)+f_n(t,x))\,dt +(\cM_n^ku_n(t,x)
+g_n^k(t,x))\,dW^k(t)                                                              \nonumber\\
&+(\cN_n^{\rho}u_n(t,x)+h_n^{\rho}(t,x))\,dB^{\rho}_n(t),
\quad (t,x)\in H_T, \,\                                                                          \label{7.22.1 with n}\\
u_n(0,x)=&u_{n0}(x)\quad x\in\bR^d,                                                  \label{1.2.1.12}
\end{align}
 where $B_n=(B_n^{\rho})$ is an
 $\bR^{d_2}$-valued continuous adapted process of finite variation on $[0,T]$.
 The operators $\cL_n$, $\cM_n^k$ and $\cN^{\rho}_n$ are of the form
$$
\cL_n=\fra_n^{ij}(t,x)D_{ij}+\fra_n^i(t,x)D_i+\fra_n(t,x),
$$
\begin{equation}                                                                      \label{form of L and M}
\cM_n^k=\frb_{n}^{ik}(t,x)D_i+\frb_{n}^{k}(t,x),
\quad \cN_n^{\rho}=\frc_{n}^{i\rho}(t,x)D_i+\frc_{n}^{\rho}(t,x)
\end{equation}
where $\fra^{ij}_n$, $\fra^{i}_n$, $\fra_n$, $\frb^{ik}_n$,  $\frb^{k}_n$,
$\frc^{i\rho}_n$ and $\frc^{\rho}_n$ are $\cP\otimes\cB(\bR^d)$-measurable
real functions on $\Omega\times H_T$
for $i,j=1,...,d$, $k=1,...,d_1$, $\rho=1,...,d_2$ and $n\geq1$. For each $n$
the initial value $u_{n0}$ is an $H^0$-valued
$\cF_{0}$-measurable random variable,
$f_n$ is an $H^{-1}$-valued predictable process and  $g^k_n$
and $h^\rho_n$ are $H^0$-valued predictable processes for $k=1,...,d_1$
and $\rho=1,...,d_2$.
We use the notation $|g_n|^2_r=\sum_k|g_n^k|^2_r$
and $|h_n|^2_r=\sum_{\rho}|h_n^{\rho}|_r^2$ for $r\geq0$.

 Let $l\geq0$ be an integer and let $K\geq0$, $\gamma>0$ be fixed constants.

 We assume the {\it stochastic parabolicity condition}.
 \begin{assumption}                                                                                             \label{assumption sp}
 There is a constant $\lambda\geq0$ such that for all $n\geq1$,
 $dP\times dt\times dx$ almost all $(\omega,t,x)\in\Omega\times H_T$
we have
 \begin{equation*}
 (\fra^{ij}_n-\tfrac{1}{2}\frb^{ik}_n\frb^{jk}_n)z^iz^j\geq \lambda|z|^2
 \quad\text{for all $z=(z^1,...,z^d)\in\bR^d$}.
 \end{equation*}
 \end{assumption}
In the case when $\lambda>0$ we will use the
 the following conditions.

 \begin{assumption}                                                             \label{assumption 1.2.1.12}
 The coefficients $\fra^{ij}_n$,
$\frb^{ik}_n$, $\frc^{i\rho}_n$ and their derivatives in $x$
 up to order $l\vee1$
 are bounded
 in magnitude by $K$ for all $i,j=1,...,d$, $k=1,...,d_1$,
$\rho=1,...,d_2$ and $n\geq1$.
  The coefficients $\fra^{i}_n$, $\fra_n$,  $\frb^{k}_n$,
$\frc^{\rho}_n$ and their derivatives in $x$ up to order $l$
are bounded in magnitude
by $K$ for all $i=1,...,d$, $k=1,...,d_1$, $\rho=1,...,d_2$ and $n\geq1$.
 \end{assumption}

\begin{assumption}                                                   \label{assumption free terms rate}
We have $\abs{u_{n0}}_l=O(n^{-\gamma})$,
$$
\int_0^T|f_n(s)|^2_{l-1}\,ds=O(n^{-2\gamma}), \quad
\int_0^T|g_n(s)|_l^2\,ds=O(n^{-2\gamma}),
$$
$$
\int_0^T|h_n^{\rho}(t)|^2_{l}d\|B^{\rho}\|(t)=O(n^{-2\gamma}), \quad
\sum_{\rho=1}^{d_2}\|B^{\rho}_n\|(T)=o(\ln n).
$$
\end{assumption}

Let $u_n$ be a generalised solution of
\eqref{7.22.1 with n}-\eqref{1.2.1.12}
in the sense of Definition \ref{Definition 2.3.3.12},
such that $u_n$ is an $H^l$-valued weakly continuous process,
$u_n(t)\in H^{l+1}$ for $P\times dt$-almost every
$(\omega,t)\in\Omega\times[0,T]$, and
almost surely
$$
\int_0^T|u_n(t)|^2_{l+1}\,dt<\infty.
$$
Then for $n\to \infty$ we have the following result.

\begin{thm}                                                                                \label{theorem:decay}
Let Assumptions \ref{assumption 1.2.1.12},
\ref{assumption free terms rate} and
\ref{assumption sp} with $\lambda>0$ hold.
Then
\begin{equation}                                                                              \label{2.3.1.12}
\sup_{t\leq T}|u_n(t)|_{l}^2+\int_0^T|u_n(t)|^2_{l+1}\,dt
=O(n^{-2\kappa})\quad
\text{a.s. for $\kappa<\gamma$}.
\end{equation}
\end{thm}

\begin{proof}
By the definition of the generalised solution
$$
(u_n(t),\varphi) = (u_{n0},\varphi)
$$
$$
+\int_0^t\big[-(\fra^{ij}_nD_iu_n(s),D_j\varphi)
 +\big((\fra^i_n-(D_j\fra_{n}^{ij}))D_iu_n(s)+\fra_nu_n(s)+f_n(s),\varphi\big)\big]\,ds                                \\
$$
$$
 \quad +\int_0^t(\cM^ku_n(s)+g_n^k(s),\varphi)\,dW^k(s)
     +\int_0^t(\cN^{\rho}_nu_n(s)+h_n^{\rho}(s),\varphi)\,dB^{\rho}_n(s)
$$
for all $\varphi\in C^{\infty}_0(\bR^d)$.  By It\^o's formula
$$
|u_n(t)|^2_0=|u_{n0}|^2_0+\int_0^t\mathcal I_n(s)\,ds
+\int_0^t\mathcal J^k_n(s)\,dW^k(s)+\int_0^t\mathcal K^{\rho}_n(s)\,dB^{\rho}(s),
$$
with
$$
\cI_n=-2(\fra^{ij}_nD_iu_n,D_ju_n)
 +2\big((\fra^i_n-\fra_{nj}^{ij})D_iu_n+\fra_nu_n+f_n,u_n\big)
 $$
 $$
 +(\frb^{ik}_nD_iu_n+\frb^k_nu_n+g_n^k,\frb^{jk}_nD_ju_n
 +\frb^k_nu_n+g_n^k),
$$
$$
\cJ^k_n=2(\cM^k_nu_n+g^k_n,u_n),
\quad \cK^{\rho}_n=2(\cN_n^{\rho}u_n+h^{\rho}_n,u_n),
$$
where $\fra_{nj}^{ij}:=D_j\fra_{n}^{ij}$.
By Assumption \ref{assumption sp}
$$
-2(\fra^{ij}_nD_iu_n,D_ju_n)+(\frb^{ik}_nD_iu_n,\frb^{jk}_nD_ju_n)
\leq -2\lambda\sum_{i=1}^d|D_iu_n|^2.
$$
Hence by standard estimates and Lemma \ref{Lemma:Gyongy} \textit{(i)},
$$
\cI_n\leq -\lambda|u_n|^2_1+C(|u_n|^2+|f_n|_{-1}^2+|g_n|^2),
\quad |\cK^{\rho}_n|\leq C(|u_n|^2+|h_n^{\rho}|^2),
$$
$$
|\cJ^k_n|^2\leq C(|u_n|^4+|g^k_n|^2|u_n|^2),
$$
with a constant $C=C(K,\lambda,d,d_1)$.
Consequently, almost surely
$$
|u_{n}(t)|^2+\lambda \int_0^t|u_n|^2_1\,ds\leq |u_{n0}|^2
+\int_0^t|u_n|^2\,dQ_n                                          \nonumber\\
$$
\begin{equation}
+\int_0^t\cJ^k_n\,dW^k+C\int_0^t\left(|f_n|_{-1}^2+|g_n|^2\right)\,ds
+C\int_0^t|h_n^{\rho}|^2\,d\|B^{\rho}_n\|          \label{2.2.1.12}
\end{equation}
for all $t\in[0,T]$ and $n\geq1$,
 where $Q_n(s)=C\left(s+\sum_{\rho=1}^{d_2}\|B^{\rho}_n\|(s)\right)$.
 By Assumption \ref{assumption free terms rate}
 we have
 $$
 Q_n(T)=o(\ln n),
 \quad |u_{n0}|^2+\int_0^T|f_n|_{-1}^2\,ds
+\int_0^T|h_n^{\rho}(s)|^2\,d\|B^{\rho}_n\|(s)=O(n^{-2\gamma})
 $$
almost surely.
Notice also that for
\begin{equation}                                                                                 \label{1.3.1.12}
 m_n(t):=\int_0^t\cJ^k_n\,dW^k(s)
\end{equation}
 we have
 $$
 d\<m_n\>=\sum_{k=1}^{d_1}|\cJ^k_n|^2\,dt
 \leq Cd_1(|u_n(t)|^4+\gamma_n|g_n|^2|u_n(t)|^2)\,dQ_n,
 $$
where $\gamma_n(t)=dt/dQ_n(t)$.
Due to Assumption \ref{assumption free terms rate}
$$
 \int_0^T\gamma_n|g_n|^2\,dQ_n
 = \int_0^T|g_n|^2\,ds=O(n^{-2\gamma})
 \quad
 \text{almost surely}.
 $$
Hence applying Lemma \ref{Lemma:gronwall} with
 $$
 y_n(t):=|u_{n}|^2+\lambda \int_0^t|u_n|^2_1\,ds,
$$
$$
 q_n(t):=|u_{n0}|^2+C\int_0^t(|f_n|_{-1}^2+|g_n|^2)\,ds
+C\int_0^t|h_n^{\rho}|^2\,d\|B^{\rho}_n\|,
 $$
and with $m_n$ defined in \eqref{1.3.1.12}, from \eqref{2.2.1.12} we
get \eqref{2.3.1.12} for $l=0$.

Assume now that $l\geq1$ and let
$\alpha$  be a multi-index such that $1\leq|\alpha|\leq l$.
Then $\alpha=\beta+\gamma$ for some multi-index $\gamma$ of length 1,
and by definition of the generalised solution we get
$$
(D^{\alpha} u_n(t),\varphi) =(D^{\alpha}u_{n0},\varphi)                                                                     \\
$$
$$
-\int_0^t\big[(D^{\alpha}\fra^{ij}_nD_iu_n,D_j\varphi)
 +\big(D^{\beta}\{(\fra^i_n-\fra_{nj}^{ij})D_iu_n+\fra_nu
 +f_n\},D^{\gamma}\varphi\big)\big]\,ds                                                                                              \\
 $$
 $$
 +\int_0^t(D^{\alpha}\cM^ku_n+D^{\alpha}g^k,\varphi)\,dW^k(s)
     +\int_0^t(D^{\alpha}\cN^{\rho}_nu_n
     +D^{\alpha}h_n^{\rho},\varphi)\,dB^{\rho}(s).
$$
Hence by It\^o's formula
$$
|D^{\alpha}u_n(t)|^2_0=|D^{\alpha}u_{n0}|^2_0
+\int_0^t\mathcal I_n^{\alpha}\,ds
+\int_0^t\mathcal K^{\rho\alpha}_n\,dB^{\rho}
+m_n^{\alpha}(t),
$$
with
$$
m_n^{\alpha}(t)=\int_0^t\mathcal J^{k\alpha}_n\,dW^k,
\quad
\cJ^{k\alpha}_n
=2(D^{\alpha}\cM^k_nu_n+D^{\alpha}g^k_n,D^{\alpha}u_n),
$$
$$
\cI_n^{\alpha}=-2(D^{\alpha}\fra^{ij}_nD_iu_n,D_jD^{\alpha}u_n)
 -2\big(D^{\beta}\{(\fra^i_n-\fra_{nj}^{ij})D_iu_n+\fra_nu+f_n\},
 D^{\alpha}D^{\gamma}u_n\big)
 $$
 $$
 +(D^{\alpha}\{\frb^{ik}_nD_iu_n+\frb^ku_n+g^k\},
 D^{\alpha}\{\frb^{jk}_nD_ju_n+\frb^k_n\frb^k_nu_n+g^k\}),
$$
$$
\cK^{\rho\alpha}_n=2(D^{\alpha}\cN_n^{\rho}u_n+D^{\alpha}h^{\rho}_n,
D^{\alpha}u_n).
$$
Due to Assumptions \ref{assumption sp}
and \ref{assumption 1.2.1.12} we get
$$
-2(D^{\alpha}\fra^{ij}_nD_iu_n,D^{\alpha}D_ju_n)+(D^{\alpha}\frb^{ik}_nD_iu_n,
D^{\alpha}\frb^{jk}_nD_ju_n)
$$
$$
\leq -\lambda\sum_{i=1}^d|D^{\alpha}D_iu_n|^2+C|u_n|^2_l
$$
with a constant $C=C(K,d,d_1,l)$.
Hence by standard estimates
$$
\cI_n^{\alpha}
\leq -\frac{\lambda}{2}
\sum_{i=1}^{d}
|D^{\alpha}D_iu_n|^2+C(|u_n|^2_l+|f_n|_{l-1}^2+|g_n|^2_l),
$$
and by Lemma \ref{Lemma:Gyongy}  \textit{(i)},
$$
|\cK^{\rho\alpha}_n|\leq C(|u_n|^2_l+|h_n^{\rho}|^2_{l}),
\quad
|\cJ^{k\alpha}_n|^2\leq C(|u_n|^4_l+|g_n^k|^2_l|u_n|^2_l),
$$
with a constant $C=C(K,\lambda,d,d_1,l)$.
Consequently, almost surely
$$
|D^{\alpha}u_{n}(t)|^2_0
+\frac{\lambda}{2}
\int_0^t\sum_{i=1}^{d}|D^{\alpha}D_iu_n|^2\,ds
\leq |D^{\alpha}u_{n0}|^2
+C\int_0^t|u_n|^2_l\,dV_n
$$
\begin{equation}                                                                                                   \label{3.3.1.12}
+C\int_0^t(|f_n|_{l-1}^2+|g_n|^2_l)\,ds
+C\int_0^t|h_n^{\rho}|^2_l\,d\|B^{\rho}_n\|+m_n^{\alpha}(t)
\end{equation}
for every $\alpha$, such that $1\leq|\alpha|\leq l$.
By virtue of \eqref{2.2.1.12} this inequality holds also for
$|\alpha|=0$. Thus summing up inequality \eqref{3.3.1.12} over
all multi-indices $\alpha$ with $|\alpha|\leq l$
we get almost surely
\begin{align}
y_n(t):=|u_n(t)|_{l}^2+\frac{\lambda}{2}\int_0^t|u_n|^2_{l+1}\,ds
&\leq
\int_0^t|u_n|^2_l\,dQ_n +M_n(t)+q_n(t)                                     \nonumber\\
&\leq \int_0^ty_n\,dQ_n +M_n(t)+q_n(t)
\end{align}
for all $t\in[0,T]$ and $n\geq1$, with
$$
M_n(t)=\sum_{|\alpha|\leq l}m_n^{\alpha}(t),
$$
$$
q_n(t)= |u_{n0}|^2_l
+C\int_0^t(|f_n|_{l-1}^2+|g_n|^2_l)\,ds
+C\int_0^t|h_n^{\rho}|^2_l\,d\|B^{\rho}_n\|(s),
$$
and a constant $C=C(K,\lambda,d,d_1,l)$. Clearly,
 $$
 d\<m_n^{\alpha}\>=\sum_{k=1}^{d_1}|\cJ^{k\alpha}_n|^2\,dt
 \leq C(|u_n|^4_l+\gamma_n|g_n|^2_l|u_n|^2_l)\,dQ_n,
 $$
so
$$
d\<M_n\>\leq C(|u_n|^4_l+\gamma_n|g_n|_l^2|u_n|^2_l)\,dQ_n
\leq C(y_n^2+\gamma_n|g_n|^2_ly_n)\,dQ_n
$$
with constants $C=C(K,\lambda,d,d_1,l)$. Hence we finish
the proof of the lemma by using
Assumption \ref{assumption free terms rate} and applying
Lemma \ref{Lemma:gronwall}.
\end{proof}

In the degenerate case, i.e., when $\lambda=0$ in Assumption \ref{assumption sp},
we need to replace Assumptions \ref{assumption 1.2.1.12} and
\ref{assumption free terms rate} by somewhat stronger assumptions in order
to have the conclusion of the previous lemma.

 \begin{assumption}                                                                                            \label{assumption 1.4.1.12}
 The coefficients $\fra^{ij}_n$ and their derivatives in $x$
 up to order $l\vee2$,
 the coefficients $\frb^{ik}_n$, $\fra^{i}_n$, $\frc^{i\rho}_n$
 and their derivatives in $x$
 up to order $l\vee1$,
 and  the coefficients $\fra_n$,  $\frb^{k}_n$,
$\frc^{\rho}_n$ and their derivatives in $x$ up to order $l$
 are $\cP\times\cB(\bR^d)$-measurable real functions,
 in magnitude bounded by $K$ for all $i,j=1,...,d$, $k=1,...,d_1$,
 $\rho=1,...,d_2$ and $n\geq1$.
 \end{assumption}

\begin{assumption}                                                                                     \label{assumption 2.4.1.12}
We have $\abs{u_{n0}}_l=O(n^{-\gamma})$,
$$
\int_0^T|f_n|^2_{l}\,ds=O(n^{-2\gamma}), \quad
\int_0^T|g_n|^2_{l+1}\,ds=O(n^{-2\gamma}),
$$
$$
\int_0^t|h_n(t)|_{l}^2\,d\|B^{\rho}\|(t)=O(n^{-2\gamma}),
\quad \sum_{\rho=1}^{d_2}\|B^{\rho}_n\|(T)=o(\ln n).
$$
\end{assumption}

\begin{thm}                                                                                \label{theorem 1.20.1.12}
Let Assumptions \ref{assumption 1.4.1.12},
\ref{assumption 2.4.1.12} and
\ref{assumption sp} (with $\lambda=0$) hold.
Let $u_n$ be an $H^{l+1}$-valued weakly
continuous generalized solution of  \eqref{7.22.1 with n}-\eqref{1.2.1.12}.
Then
\begin{equation}                                                                              \label{2.20.1.12}
\sup_{t\leq T}|u_n(t)|_{l}=O(n^{-\kappa})\quad
\text{a.s. for each $\kappa<\gamma$}.
\end{equation}
\end{thm}
\begin{proof}
Let
$\alpha$  be a multi-index such that $|\alpha|\leq l$.
Then, as in the proof of the previous theorem, by It\^o's formula we have
\begin{equation}                                                                                 \label{4.4.1.12}
|D^{\alpha}u_n(t)|^2=|D^{\alpha}u_{n0}|^2
+\int_0^t\mathcal I_n^{\alpha}\,ds
+\int_0^t\mathcal K^{\rho\alpha}_n\,dB^{\rho}
+m_n^{\alpha}(t),
\end{equation}
with
$$
m_n^{\alpha}(t)=\int_0^t\mathcal J^{k\alpha}_n\,dW^k,
\quad
\cJ^{k\alpha}_n
=2(D^{\alpha}\cM^k_nu_n+D^{\alpha}g^k_n,D^{\alpha}u_n),
$$
$$
\cI_n^{\alpha}=-2(D^{\alpha}\fra^{ij}_nD_iu_n,D_jD^{\alpha}u_n)
 -2\big(D^{\beta}\{(\fra^i_n-\fra_{nj}^{ij})D_iu_n+\fra_nu+f_n\},
 D^{\alpha}D^{\gamma}u_n\big)
 $$
 $$
 +(D^{\alpha}\frb^{ik}_nD_iu_n+D^{\alpha}\frb^k_nu_n+D^\alpha g_n^k,
 D^{\alpha}\frb^{jk}_nD_ju_n+D^{\alpha}\frb^k_nu_n+D^\alpha g_n^k),
$$
$$
\cK^{\rho\alpha}_n=2(D^{\alpha}\cN_n^{\rho}u_n+D^{\alpha}h^{\rho}_n,
D^{\alpha}u_n),
$$
where $\beta$ and $\gamma$ are multi-indices such that $\alpha=\beta+\gamma$ and
$|\gamma|=1$ if $|\alpha|\geq1$.
By \cite[Lemma~2.1]{KR} and \cite[Remark~2.1]{KR},
$$
\cI_n^{\alpha}\leq C(|u_n|_{l}^2+|f_n|^2+|g_n|^2_{l+1}),
$$
and  by  Lemma \ref{Lemma:Gyongy}  \textit{(i)},
$$
|\cK^{\rho\alpha}_n|\leq C(|u_n|^2_l+|h_n^{\rho}|^2_{l})
$$
with a constant $C=C(K,d,d_1,l)$. Thus from \eqref{4.4.1.12} we get
$$
|D^{\alpha}u_{n}(t)|^2
\leq |D^{\alpha}u_{n0}|^2
+\int_0^t|u_n|^2_l\,dQ_n
$$
\begin{equation}                                                                                                   \label{5.4.1.12}
+C\int_0^t(|f_n|_{l}^2+|g_n|^2_{l+1})\,ds
+C\int_0^t|h_n^{\rho}|^2_l\,d\|B^{\rho}_n\|+m_n^{\alpha}(t)
\end{equation}
for $|\alpha|\leq l$,
where $Q_n(s)=C\left(s+\sum_{\rho=1}^{d_2}\|B^{\rho}_n\|(s)\right)$.
Summing up these inequalities over $\alpha$, $|\alpha|\leq l$,
we obtain
\begin{equation}                                                                                                     \label{6.4.1.12}
y_n(t):=|u_n(t)|_{l}^2
\leq \int_0^t|u_n|^2_l\,dQ_n +M_n(t)+q_n(t)
\end{equation}
for all $t\in[0,T]$ and $n\geq1$, where
$$
M_n(t)=\sum_{|\alpha|\leq l}m_n^\alpha(t),
$$
$$
q_n(t)= |u_{n0}|^2_l
+C\int_0^t(|f_n|_{l}^2+|g_n|^2_{l+1})\,ds
+C\int_0^t|h_n^{\rho}|^2_l\,d\|B^{\rho}_n\|,
$$
and $C=C(K,d,d_1,l)$ is a constant. Hence the rest of the proof
is the same as that in the proof of the previous theorem.
\end{proof}

\section{Proof of the main theorems}                                                                  \label{section main}

To prove our main results we look for processes $r_n$
such that
\begin{equation}                                                                                                  \label{10.22.1.12}
\sup_{t\leq T}|r_n(t)|_m=O(n^{-\kappa}) \quad\text{a.s.
for each $\kappa<\gamma$},
\end{equation}
and
$
v_n:=u-u_n-r_n
$
solves a suitable Cauchy problem of the type
\eqref{7.22.1 with n}-\eqref{1.2.1.12},
satisfying the conditions of Theorem \ref{theorem:decay}
or Theorem \ref{theorem 1.20.1.12},
so that we could get for each $\kappa<\gamma$
$$
\sup_{t\leq T}|v_n|_{m}=O(n^{-\kappa})\quad\text{a.s. for each $\kappa<\gamma$}.
$$

\subsection{Proof of Theorem \ref{theorem main1}}
We will carry out the strategy above in several steps,
formulated as lemmas below.
By a well-known result, see, e.g., \cite{R2},
$u$ is an $H^{m+1}$-valued strongly continuous process, and
\begin{equation}                                                                             \label{3.22.1.12}
\sup_{t\leq T}|u|^2_{m+1}+\int_0^T|u|^2_{m+2}\,dt<\infty
\quad\text{almost surely}.
\end{equation}
Moreover, we can apply Theorem \ref{theorem growth nondeg}
with $l=m+3$ to get
\begin{equation}                                                                            \label{2.22.1.12}
\sup_{t\leq T}|u_n|^2_{m+3}+\int_0^T|u_n|^2_{m+4}\,dt
=O(n^{\varepsilon})\quad\text{a.s. for any $\varepsilon>0$}.
\end{equation}
Notice that $u-u_n$ satisfies
\begin{align}
d(u-u_n) = &\{L_n(u-u_n)+\bar f_n\}\,dt
+\{M^k_n(u-u_n)+\bar g_n^k\}\,dW^k                                                 \nonumber \\
&+\tfrac{1}{2}\{M^kM^ku+M^kg^k\}\,dt+(M^k_nu_n+g^k_n)\,d(W^k-W^k_n),                                              \label{4.22.1.12}
\end{align}
with
$$
\bar f_n:=f-f_n+(L-L_n)u, \quad \bar g^k_n:=g^k-g^k_n+(M^k-M^k_n)u.
$$
Notice also
that due to \eqref{3.22.1.12}
and Assumption \ref{assumption 06.02.08}
we have
\begin{equation}                                                                              \label{5.22.1.12}
\int_0^T(|\bar f_n|_{m-1}^2+|\bar g_n|^2_m)\,dt=O(n^{-2\gamma}).
\end{equation}
Next we rewrite equation \eqref{4.22.1.12} as an equation  for
$$
w_n=u-u_n-z_n,\quad \text{where $z_n=(M^k_nu_n+g^k_n)(W^k-W^k_n)$}.
$$
Note that by \eqref{2.22.1.12} and by our assumptions we have
for each $\kappa<\gamma$
\begin{equation}                                                                             \label{6.6.3.12}
\sup_{t\leq T}|z_n(t)|_m^2+\int_0^T|z_n(t)|_{m+1}^2\,dt
=O(n^{-2\kappa})\quad \text{almost surely}.
\end{equation}
Set $\cL_n:=L_n+\frac{1}{2}M_n^kM^k_n$ and recall the definition of
$S_n^{kl}$ in Remark \ref{Rem:Sn and An}.

\begin{lem}                                                                       \label{lemma eq vn 1}
The process $w_n$ solves
\begin{align}
dw_n=&(\cL_nw_n+F_n)\,dt+(M^k_nw_n+G_n^k)\,dW^k                                      \nonumber\\
&-M^k_n(M^l_nu_n+g^l_n)\,dS_n^{kl},                                                   \label{eq w1}
\end{align}
where $G^k_n=\bar g^k_n+M^k_nz_n$ and
\begin{align*}
F_n=&\bar f_n+\tfrac{1}{2}(M^kM^k-M^k_nM^k_n)u
+\tfrac{1}{2}(M^kg^k-M^k_ng^k_n)                                                       \nonumber\\
&-(M^k_n(L_nu_n+f_n))(W^k-W^k_n)+\cL_nz_n.                                               \nonumber
\end{align*}
\end{lem}

\begin{proof} By using It\^{o}'s formula one can easily verify that
\begin{align}
    dz_n
    =&M^k_n(L_nu_n+f_n)(W^k-W^k_n)\,dt
     +M^k_n(M^l_nu_n+g^l_n)(W^k-W_n^k)\,dW_n^l                                \nonumber\\
     &+(M^k_nu_n+g^k_n)\,d(W^k-W_n^k).                                                     \label{3.20.1.12}
\end{align}
Hence
\begin{align*}
    dw_n = &\{L_n(u-u_n)+\bar f_n+\tfrac{1}{2}M^kM^ku
    +\tfrac{1}{2}M^kg^k                                                              \\
&-M^k_n(L_nu_n+f_n)(W^k-W^k_n)\}\,dt                                                  \\
        &+\{M^k_n(u-u_n)+\bar g_n^k\}\,dW^k
        -M^k_n(M^l_nu_n+g^l_n)(W^k-W^k_n)\,dW^l_n                                      \\
        = &\{\cL_n(u-u_n)+\bar f_n+\tfrac{1}{2}(M^kM^k-M^k_nM^k_n)u
        +\tfrac{1}{2}(M^kg^k-M^k_ng^k_n)\}\,dt                                                                                    \\
        &-(M^k_nL_nu_n+f_n)(W^k-W^k_n))\,dt \\
        &+(M^k_nw_n+G_n^k)\,dW^k
       -M^k_n(M^l_nu_n+g^l_n)\,dS_n^{kl}                                             \\
        = &(\cL_nw_n+F_n)\,dt
        +(M^kv_n+G_n^k)\,dW^k-M^k_n(M^l_nu_n+g^l_n)\,dS^{kl}_n,
\end{align*}
The lemma is proved.
\end{proof}
It is easy to show that
due to \eqref{5.22.1.12}, \eqref{2.22.1.12}, \eqref{3.22.1.12},
Assumptions  \ref{assumption 06.02.07},
\ref{assumption 1.22.1.12}, \ref{assumption 06.02.08},
and \ref{assumption 15.04.06} $(i)$
we have
\begin{equation}                                                                              \label{13.7.3.12}
\int_0^T(|F_n|_{m-1}^2+|G_n|^2_m)\,dt=O(n^{-2\kappa})
\quad\text{(a.s.) for each $\kappa<\gamma$}.
\end{equation}

We rewrite the last term in the right-hand side of
{\eqref{eq w1} into symmetric
and antisymmetric parts as follows:
$$
    M^k_n(M^l_nu_n+g^l_n)\,dS_n^{kl}
    = \tfrac{1}{2}(M^k_nM^l_n+M^l_nM^k_n)u_n\,dS_n^{kl}
$$
$$
+\tfrac{1}{2}(M^k_nM^l_n-M_n^lM^k_n)u_n\,dS_n^{kl}                                                 \\
        +\tfrac{1}{2}M^k_ng^l_n\,d(S_n^{kl}+S_n^{lk})
+\tfrac{1}{2}(M^l_ng^k_n-M^k_ng^l_n)\,dS_n^{lk}
$$
$$
=\tfrac{1}{2}M^k_n(M^l_nu_n+g^l_n)\,d(S_n^{kl}+S_n^{lk})                                 \\
+\tfrac{1}{2}\left([M^l_n,M^k_n]u_n+M^k_ng^l_n-M^l_ng^k_n\right)\,dS_n^{kl},
$$
where $[A,B]=BA-AB$. Thus using Remark \ref{Rem:Sn and An} we get
$$
M^k_n(M^l_nu_n+g^l_n)\,dS_n^{kl} = -\tfrac{1}{2}M^k_n(M^l_nu_n+g^l_n)\,dq_n^{kl}
$$
$$
+\tfrac{1}{2}M^k_n(M^l_nu_n+g^l_n)\,d(R_n^{kl}+R_n^{lk})
+\tfrac{1}{2}\left([M^l_n,M^k_n]u_n+M^k_ng^l_n-M^l_ng^k_n\right)\,dS_n^{kl}
$$
$$
= -\tfrac{1}{2}M^k_n(M^l_nu_n+g^l_n)\,dq_n^{kl}
+\tfrac{1}{2}
\left(M^l_n\diamond
M^k_nu_n+M^kg^l_n+M^l_ng^k_n\right)\,dR_n^{kl}
$$
\begin{equation}                                                                    \label{bad term1}
+\tfrac{1}{2}\left([M^l_n,M^k_n]u_n
+M^k_ng^l_n-M^l_ng^k_n\right)\,dS_n^{kl},
\end{equation}
where we use the notation
$A\diamond B=BA+AB$ for linear operators $A$ and $B$.
Thus  equation \eqref{eq w1} can be rewritten as follows.

\begin{lem}                                                                                                           \label{lemma 4.20.1.12}
The process $w_n$ solves
$$
dw_n =(\cL_nw_n+F_n)\,dt
+(\tfrac{1}{2}[M^l_n,M^k_n]w_n+H_n^{kl})\,dS_n^{kl}
+(M^k_nw_n+\bar{G}_n^k)\,dW^k                                                           \nonumber \\
$$
\begin{equation}
+\tfrac{1}{2}M^k_n(M^l_nu_n+g^l_n)\,dq_n^{kl}
+\tfrac{1}{2}\left([M^k,M^l]u+M^lg^k-M^kg^l\right)\,dS_n^{kl},           \label{eq w2}
\end{equation}
where
\begin{align*}
H_n^{kl}=&\tfrac{1}{2}[M^l_n,M^k_n](M^r_nu_n+g^r_n)(W^r-W^r_n)                        \\
&+\tfrac{1}{2}([M^k_n,M^l_n]-[M^k,M^l])u
+\tfrac{1}{2}(M^l_ng^k_n-M^lg^k+M^kg^l-M^k_ng^l_n),\\
\bar{G}_n^k=&G_n^k-\tfrac{1}{2}
\left(M^k_n\diamond M^l_nu_n+M^l_ng^k_n+M^k_ng^l_n\right)
(W^l-W_n^l)
\end{align*}
\end{lem}

\begin{proof}
Plugging \eqref{bad term1} into \eqref{eq w1} we get
\begin{align*}
    dw_n = &(\cL_nw_n+F_n)\,dt+(M^k_nw_n+\bar
G_n^k)\,dW^k                                                                      \\
&+\tfrac{1}{2}M^k_n(M^l_nu_n+g^l_n)\,dq_n^{kl}
-\tfrac{1}{2}\{[M^l_n,M^k_n]u_n+M^k_ng^l_n-M^l_ng^k_n\}\,dS_n^{kl}                      \\
=&(\cL_nw_n+F_n)\,dt+(M^k_nw_n+\bar G_n^k)\,dW^k                                   \\
     &+\tfrac{1}{2}M^k_n(M^l_nu_n+g^l_n)\,dq_n^{kl}
+\left(\tfrac{1}{2}[M^l_n,M^k_n]w_n+M^l_ng^k_n-M^k_ng^l_n\right)\,dS_n^{kl}           \\
     &-\tfrac{1}{2}[M^l_n,M^k_n]u\,dS_n^{kl}
     +\tfrac{1}{2}[M^l_n,M^k_n](M^r_nu_n+g^r_n)(W^r-W_n^r)\,dS_n^{kl}                   \\
     = &(\cL_nw_n+F_n)\,dt+(\tfrac{1}{2}[M^l_n,M^k_n]w_n+H_n^{kl})\,dS_n^{kl}
     +(M^k_nw_n+\bar{G}_n^k)\,dW^k                                           \\
     &+\tfrac{1}{2}M^k_n(M^l_nu_n+g^l_n)\,dq_n^{kl}+\tfrac{1}{2}
     \{[M^k,M^l]u+M^lg^k-M^kg^l\}\,dS_n^{kl}.
\end{align*}
\end{proof}

In the same way as \eqref{13.7.3.12} is proved,  we can easily
get
\begin{equation}                                                                              \label{6.22.1.12}
\int_0^T|\bar G_n|_{m-1}^2\,dt=O(n^{-2\kappa}),
\quad \sup_{t\leq T}|H_n^{kl}(t)|^2_m=O(n^{-2\kappa})
\quad\text{for $\kappa<\gamma$ }
\end{equation}
almost surely for all $k,l=1,...,d_1$.
Finally we rewrite \eqref{eq w2} as an equation for $v_n=w_n-r_n$, where
\begin{equation*}                                                                                                            \label{eq r}
r_n=\tfrac{1}{2}M_n^k(M^l_nu_n+g^l_n)q_n^{kl}+\tfrac{1}{2}
\{[M^k,M^l]u+M^kg^l-M^lg^k\}S_n^{kl}.
\end{equation*}
Notice that by \eqref{2.22.1.12} and Remark \ref{Rem:Sn and An},
$r_n$ satisfies \eqref{6.6.3.12} in place of $z_n$.
\begin{lem}                                                                                                \label{lemma 4.3.3.12}
The process $v_n$ solves
\begin{align}
    dv_n =& (\cL_nv_n+\widetilde{F}_n)\,dt
    +(\tfrac{1}{2}[M^k_n,M^l_n]v_n+\widetilde{H}_n^{kl})\,dS_n^{kl}
+(M^k_nv_n+\widetilde{G}_n^k)\,dW^k                                                 \nonumber\\
     & -\tfrac{1}{2}M^k_nM^l_n(M^r_nu_n+g^r_n)(W^k-W_n^k)\,dB_n^{lr},           \label{9.22.1.12}
\end{align}
where $B_n^{lr}$ is as in \eqref{def Bn} and
\begin{align*}
\tilde{F}_n=&F_n+\cL_nr_n-\tfrac{1}{2}M^k_nM^l_n(L_nu_n+f_n)q_n^{kl}
-\tfrac{1}{2}[M^k,M^l](\cL u+\tfrac{1}{2}M^rg^r+f)S_n^{kl},                    \\
\widetilde{G}_n^k=&\bar G_n+M^k_nr_n
-\tfrac{1}{2}[M^r,M^l](M^ku+g^k)S_n^{rl},
\\
\widetilde{H}_n^{kl}=&H_n^{kl}+\tfrac{1}{2}[M^l_n,M^k_n]r_n.
\end{align*}
\end{lem}

\begin{proof}
Indeed,
\begin{align*}
    dv_n
    =& (\cL_n v_n+F_n)\,dt
    +(\tfrac{1}{2}[M_n^l,M^k_n]v_n+\tilde H_n^{kl})\,dS_n^{kl}
    +(M^k_nv_n+\bar{G}_n^k)\,dW^k                                                                              \\
    &+\cL_nr_n\,dt+M^k_nr_n\,dW^k                                                                               \\
     &-\tfrac{1}{2}M^k_nM^l_n(L_nu_n+f_n)q_n^{kl}\,dt
     -\tfrac{1}{2}M^k_nM^l_n(M^r_nu_n+g^r_n)q_n^{kl}\,dW_n^r                                    \\
     &-\tfrac{1}{2}[M^k_n,M^l_n](\cL u+\tfrac{1}{2}M^rg^r+f)S_n^{kl}\,dt
-\tfrac{1}{2}[M^k_n,M^l_n](M^ru+g^r)S_n^{kl}\,dW^r                                                       \\
     &= (\cL v_n+\tilde F_n)\,dt
+(\tfrac{1}{2}[M^l_n,M^k_n]v_n+\tilde H_n^{kl})\,dS_n^{kl}
     +(M^k_n v_n+\tilde{G}_n^k)\,dW^k \\
     &-\tfrac{1}{2}M^k_nM^l_n(M^r_nu_n+g^r_n)(W^k-W^k_n)\,dB_n^{lr}.
     \end{align*}
\end{proof}
Making use of \eqref{13.7.3.12}, \eqref{6.22.1.12} and \eqref{7.22.1.12},
we easily obtain  that for $\kappa<\gamma$
\begin{equation}                                                                                     \label{8.22.1.12}
\int_0^T(|\tilde F_n|^2_{m-1}+|\tilde G_n^k|_{m}^2)\,dt=O(n^{-2\kappa}),
\quad \sup_{t\leq T}|\tilde H_n^{kl}(t)|^2_m=O(n^{-2\gamma})
\end{equation}
almost surely for $k,l=1,...,d_1$.
Hence we finish the proof of the theorem by applying
Theorem \ref{theorem:decay}
with $l=m$ to equation \eqref{9.22.1.12} and using \eqref{10.22.1.12} for $z_n$ and $r_n$.
\hfill $\square$

\subsection{Proof of Theorem \ref{theorem main2}}
We follow the proof of Theorem \ref{theorem main1} with the
necessary changes. By a well-known
theorem on degenerate stochastic PDEs
from \cite{KR},
$u$ is an $H^{m+2}$-valued
weakly continuous process,
and by Theorem \ref{theorem growth deg}
with $l=m+4$ we have
\begin{equation}                                                                            \label{4.3.3.12}
\sup_{t\leq T}|u_n|^2_{m+4}
=O(n^{\varepsilon})\quad\text{a.s. for $\varepsilon>0$}.
\end{equation}
Clearly, $u-u_n$ satisfies equation \eqref{4.22.1.12},
and
\begin{equation}                                                                              \label{5.3.3.12}
\int_0^T(|\bar f_n|_{m}^2+|\bar g_n|^2_{m+1})\,dt=O(n^{-2\gamma}).
\end{equation}
Moreover, Lemmas \ref{lemma eq vn 1}, \ref{lemma 4.20.1.12}
and \ref{lemma 4.3.3.12}
remain valid, and
due to  \eqref{4.3.3.12}, \eqref{5.3.3.12} and our assumptions,
we have for each $\kappa<\gamma$
\begin{equation}                                                                              \label{6.3.3.12}
\int_0^T(|F_n|_{m}^2+|G_n|^2_{m+1})\,dt=O(n^{-2\kappa}),
\end{equation}
\begin{equation}                                                                              \label{7.3.3.12}
\int_0^T|\bar G_n|_{m+1}^2\,dt=O(n^{-2\kappa}),
\quad \sup_{t\leq T}|H_n^{kl}(t)|^2_m=O(n^{-2\kappa}),
\end{equation}
\begin{equation}                                                                                     \label{8.3.3.12}
\int_0^T|\tilde F_n|^2_{m}+|\tilde G_n|_{m+1}^2\,dt=O(n^{-2\kappa}),
\quad \sup_{t\leq T}|\tilde H_n^{kl}(t)|^2_m=O(n^{-2\kappa})
\end{equation}
almost surely for $k,l=1,...,d_1$.
Note also that $r_n$ and $z_n $ satisfy \eqref{10.22.1.12}.
Hence we finish the proof of the theorem by applying
Theorem \ref{theorem 1.20.1.12}
with $l=m$ to equation \eqref{9.22.1.12}.
\hfill $\square$

Now we prove our main results in the case when the coefficients
and the free terms depend on $t$.

\subsection{Proof of Theorem \ref{theorem main1t}}
We follow the proof of Theorem  \ref{theorem main1} with the necessary
changes. As before, \eqref{3.22.1.12} and \eqref{2.22.1.12} hold.
Now $u-u_n$ satisfies equation \eqref{4.22.1.12} with an additional term,
$$
\frac{1}{2}\sum_{k=1}^{d_1}\big(M^{k(k)}u+g^{k(k)}\big)\,dt,
$$
added to the right-hand side of \eqref{4.22.1.12}. Thus to get the analogue
of Lemma \ref{lemma eq vn 1} we set
$$
N_n=\frac{1}{2}\sum_{k=1}^{d_1}M_n^{k(k)},\quad
\bar{\mathcal L}_n=\mathcal L_n+N_n
=L_n+\frac{1}{2}\Big(M^{k}_nM^{k}_n+\sum_{k=1}^{d_1}M_n^{k(k)}\Big),
$$
$$
M_n^{k(l)}=b^{ik(l)}_nD_i+b^{k(l)}_n, \quad M^{kl}=b^{ik(l)}D_i+b^{k(l)}
$$
for $k=1,...,d_1$, $l=0,...,d_1$ and $n\geq1$. Then for
\begin{equation}                                                                        \label{12.7.3.12}
w_n=u-u_n-z_n,\quad z_n=(M_n^k u_n+g^k_n)(W^k-W^k_n)
\end{equation}
the corresponding lemma reads as follows.

\begin{lem}                                                                       \label{lemma eq vn 1t}
The process $w_n$ solves
\begin{align*}
dw_n=&(\bar\cL_nw_n+\bar F_n)\,dt+(M^k_nw_n+G_n^k)\,dW^k                                      \nonumber\\
&-M^k_n(M^l_nu_n+g^l_n)\,dS_n^{kl}
-(M^{k(l)}_nu_n+g^{k(l)}_n)\,dS_n^{kl},
\end{align*}
where
\begin{align}
\bar F_n=&F_n+\frac{1}{2}\sum_{k=1}^{d_1}(M^{k(k)}-M_n^{k(k)})u
+\frac{1}{2}\sum_{k=1}^{d_1}(g^{k(k)}-g_n^{k(k)})                                           \nonumber\\
&-(M^{k(0)}u_n+g^{k(0)})(W^k-W^k_n)
+N_nz_n,                                                                                                                         \label{6.5.3.12}
\end{align}
and $F_n$ and $G_n$ are defined in Lemma \ref{lemma eq vn 1}.
\end{lem}

\begin{proof}
We need only notice that for $z_n$ equation \eqref{3.20.1.12}
holds with a new term,
$$
(M^{k(0)}_nu_n+g^{k(0)})(W^k-W_n^k)\,dt
+(M^{k(l)}_nu_n+g^{k(i)}_n)(W^{k}-W^{k}_n)\,dW^l_n?,
$$
added to its right-hand side.
\end{proof}
Hence we get the following modification of Lemma \ref{lemma 4.20.1.12}

\begin{lem}                                                                             \label{lemma 2.4.3.12}
The process $w_n$ solves
$$
dw_n =(\bar\cL_nw_n+\bar F_n)\,dt
+(\tfrac{1}{2}([M^l_n,M^k_n]+M_n^{k(l)})w_n+\bar H_n^{kl})\,dS_n^{kl}
$$
$$
+(M^k_nw_n+\bar{G}_n^k)\,dW^k+\tfrac{1}{2}M^k_n(M^l_nu_n+g^l_n)\,dq_n^{kl}
$$
\begin{equation}
+\tfrac{1}{2}\left(([M^k,M^l]-M^{k(l)})u+M^lg^k-M^kg^l-g^{k(l)}\right)\,dS_n^{kl},           \label{eq w2t}
\end{equation}
where
\begin{align}
\bar H_n^{kl}=&H^{kl}_n+\tfrac{1}{2}M_n^{k(l)}(M^r_nu_n+g^r_n)(W^r-W^r_n)          \nonumber   \\
&+\tfrac{1}{2}(M^{k(l)}-M^{k(l)}_n)u
+\tfrac{1}{2}(g^{k(l)}-g_n^{k(l)}),                                                                             \label{7.5.3.12}
\end{align}
and $H_n^{kl}$ and $\bar{G}_n^k$ are defined in Lemma \ref{lemma 4.20.1.12}.
\end{lem}
Now we rewrite equation \eqref{eq w2t} as an equation for $\bar v_n:=w_n-\bar r_n$,
where
\begin{equation}                                                                                                       \label{11.7.3..12}
\bar r_n=r_n-\tfrac{1}{2}(M^{k(l)}u+g^{k(l)})S_n^{kl}
\end{equation}
$$
=\tfrac{1}{2}M_n^k(M^l_nu_n+g^l_n)q_n^{kl}
+\tfrac{1}{2}(([M^k,M^l]-M^{k(l)})u+M^lg^k-M^kg^l-g^{k(l)})S_n^{kl}.
$$
To this end we set
$$
\tilde f=f+\tfrac{1}{2}M^kg^k+\tfrac{1}{2}\sum_{k=1}^{d_1}g^{k(k)},
$$
and notice that
\begin{align*}
dM_n^k(M_n^lu_n+g^l_n)
=&M_n^kM^l_n(L_nu_n+f_n)\,dt+M_n^kM^l_n(M^j_nu_n+g^j_n)\,dW^j_n\\
&+T_n^{kl0}\,dt+T_n^{klj}\,dW^j_n,
\end{align*}
where
$$
T_n^{kl0}=(M_n^{k(0)}M_n^l+M_n^{k}M_n^{l(0)})u_n+M_n^{k(0)}g^{l}_n+M_n^kg^{l(0)}_n,
$$
$$
T_n^{klj}=(M_n^{k(j)}M_n^l+M_n^{k}M_n^{l(j)})u_n+M_n^{k(j)}g^{l}_n+M_n^kg^{l(j)}_n.
$$
Similarly,
\begin{align*}
d[M^k,M^l]u=&[M^k,M^l](\cL u+f+\tfrac{1}{2}M^jg^j)\,dt+[M^k,M^l](M^j+g^j)\,dW^j\\
&+P^{kl0}\,dt+P^{klj}\,dW^j,
\end{align*}
$$
d(-M^{k(l)}u+M^lg^k-M^kg^l-g^{k(l)})=U^{kl0}\,dt+U^{klj}\,dW^j,
$$
where $\cL=L+\tfrac{1}{2}M^kM^k$,
\begin{align*}
P^{kl0}=&[M^k,M^{l(0)}]u+[M^{k(0)},M^{l}]u+\sum_{j=1}^{d_1}[M^{k(j)},M^{l(j)}]u\\
&+([M^k,M^{l(j)}]+[M^{k(j)},M^{l}])(M^ju+g^j)\\
&+
\tfrac{1}{2}\sum_{j=1}^{d_1}[M^k,M^l](M^{j(j)}u+g^{j(j)}),\\
P^{klj}=&([M^k,M^{l(j)}]+[M^{k(j)},M^{l}])u,\\
U^{kl0}=&-M^{k(l0)}u-M^{k(l)}(\bar\cL u+\tilde f)-M^{k(lj)}(M^ju+g^j)\\
&+M^{l(0)}g^k+M^{l}g^{k(0)}+M^{l(j)}g^{k(j)}                 \\
&-M^{k(0)}g^l-M^{k}g^{l(0)}-M^{k(j)}g^{l(j)}-g^{k(l0)}                      ,\\
U^{klj}=&-M^{k(lj)}u-M^{k(l)}(M^j u+g^j)
+M^{l(j)}g^k+M^{l}g^{k(j)}                \\
&-M^{k(j)}g^l-M^{k}g^{l(j)}-g^{k(lj)}.                       \\
\end{align*}
Let $\tilde F_n$, $\tilde G^k_n$ and $\tilde H_n^{kl}$ be defined now as
in Lemma \ref{lemma 4.3.3.12}, but with $F_n$ and $G_n^k$ replaced
 there by
$\bar F_n$ and $\bar G_n^k$ in \eqref{6.5.3.12} and \eqref{7.5.3.12}, respectively.

Thus we have the following modification of Lemma \ref{lemma 4.3.3.12}.
\begin{lem}                                                                                                \label{lemma 12.8.3.12}
The process $\bar v_n=w_n-\bar r_n$ solves
\begin{align}
    d\bar v_n =& (\bar \cL_n\bar v_n+\hat{F}_n)\,dt
    +(\tfrac{1}{2}([M^k_n,M^l_n]+M^{k(l)})\bar v_n+\hat{H}_n^{kl})\,dS_n^{kl} \nonumber\\
&+(M^k_nv_n+\hat{G}_n^k)\,dW^k
      -\tfrac{1}{2}T^{klj}(W^k-W_n^k)\,dB_n^{lr},           \label{7.6.3.12}
\end{align}
where
\begin{align*}
\hat{F}_n=&\tilde F_n+N_nr_n-\tfrac{1}{2}\bar\cL_n(M^{k(l)}u+g^{k(l)})S_n^{kl}    \\
&-\tfrac{1}{2}T^{kl0}q^{kl}_n-\tfrac{1}{2}(P_n^{kl0}+U^{kl0}_n)S_n^{kl},                      \\
\hat{G}_n^k=&\widetilde G_n-\tfrac{1}{2}M^k_n(M^{j(l)}u+g^{j(l)})S_n^{jl}
-\tfrac{1}{2}(P^{jlk}+U^{jlk})S^{jl}_n
\\
\hat {H}_n^{kl}=&\tilde H_n^{kl}+\tfrac{1}{2}M^{k(l)}r_n-
\tfrac{1}{4}([M^l_n,M^k_n]+M^{kl}_n)(M^{k(l)}u+g^{k(l)})S_n^{kl}.
\end{align*}
\end{lem}
We can verify that \eqref{8.22.1.12} holds with $\hat F_n$, $\hat G_n^k$
and $\hat H^{kl}$ in place of $\tilde F_n$, $\tilde G_n^k$ and
$\tilde H^{kl}$, respectively.
We can also see that $z_n$ and ${\bar r}_n$ satisfy \eqref{6.6.3.12}.
Hence we finish the proof by applying Theorem \ref{theorem:decay}
with $l=m$ to equation \eqref{7.6.3.12}.
\subsection{Proof of Theorem \ref{theorem 5.6.3.12}}
We get Lemma \ref{lemma 12.8.3.12} in the same way as Lemma \ref{lemma 4.3.3.12} is proved, and we can also
see that
\begin{equation*}                                                                                     \label{11.6.3.12}
\int_0^T(|\hat F_n|^2_{m}+|\hat G_n^k|_{m+1}^2)\,dt=O(n^{-2\kappa}),
\quad \sup_{t\leq T}|\hat  H_n^{kl}(t)|^2_m=O(n^{-2\kappa})
\end{equation*}
for each $\kappa<\gamma$, almost surely for $k,l=1,...,d_1$, where
$\hat F_n$, $ \hat G_n^k$ and $\hat  H_n^{kl}$ are defined in
Lemma \ref{lemma 4.3.3.12}. We can also verify that for $z_n$
and $\bar r_n$, defined in \eqref{12.7.3.12} and \eqref{11.7.3..12},
we have
$$
\sup_{t\leq T}|z_n(t)|_m+\sup_{t\leq T}|\bar r_n(t)|_m=O(n^{-\kappa})
\quad\text{for each $\kappa<\gamma$}.
$$
Hence we obtain the theorem by applying Theorem \ref{theorem 1.20.1.12}
with $l=m$ to equation \eqref{7.6.3.12}.

\bigskip

\small{\textbf{Acknowledgments.}
The first author is grateful to the organisers of the
``Seventh Seminar on Stochastic Analysis, Random Fields and Applications"
for the invitation and for the possibility of giving a talk on the results of this paper,
which were obtained while the authors were staying
in the Isaac Newton Institute of Cambridge University in April 2010.
The second author is also grateful to the University of Edinburgh's School of
Mathematics for their kind hospitality during several visits.}

\end{document}